\documentclass[12pt]{article}

\usepackage{microtype} 

\usepackage{graphicx}
\usepackage{amsmath}%
\usepackage{mathtools} 
\usepackage{amsthm}
\usepackage{amsfonts}
\usepackage{amssymb}
\usepackage{fullpage} 
\usepackage{color}
\usepackage{latexsym}
\usepackage{algorithm}
\usepackage[noend]{algorithmic}
\usepackage{subfigure}
\usepackage{soul}
\usepackage{enumerate}
\usepackage{multirow}
\usepackage[nameinlink]{cleveref}
\usepackage{thm-restate}
\usepackage{xfrac}

\usepackage{diagbox}
\usepackage{makecell}
\usepackage{array,booktabs}
\usepackage{hhline}
\usepackage[table]{xcolor}

\newtheorem{theorem}{Theorem}

\newtheorem{corollary}[theorem]{Corollary}
\newtheorem{lemma}[theorem]{Lemma}

\newtheorem{remark}[theorem]{Remark}
\newtheorem{proposition}[theorem]{Proposition}

\newtheorem{example}[theorem]{Example}

\theoremstyle{remark}

\newcommand{\old}[1]{{{}}}

\graphicspath{{./Figures/}}

\def\Re{\rm{Re}}
\def\x{\bf {x}}

\usepackage{xspace}
\usepackage[numbers,sort&compress]{natbib}

\usepackage[margin=1in]{geometry}

\newcommand{\Eq}[1]{(\ref{#1})}

\def\Thanks#1{\gdef\thefootnote{\arabic{footnote}}\thanks{#1}}
\def\ThanksComma#1{\gdef\thefootnote{\arabic{footnote},}\thanks{#1}{
}}

\begin{document}
\newenvironment{Proof}[1]{\par\noindent{\bf Proof{#1}:}\quad}{} 

\title{Maximum of Exponential Random Variables, Hurwitz's Zeta Function, and the Partition Function}

\author{D. Barak-Pelleg\Thanks{
Department of Mathematics, Ben-Gurion
University, Beer Sheva 84105, Israel.
E-mail: dinabar@post.bgu.ac.il}
 \and  D.~Berend\ThanksComma{
Departments of Mathematics and Computer Science, Ben-Gurion
University, Beer Sheva 84105, Israel.
E-mail: berend@math.bgu.ac.il}
\Thanks{Research supported in part by the Milken
Families Foundation Chair in Mathematics.}
\and
G.~Kolesnik\ThanksComma{
Department of Mathematics, California State University, Los
Angeles,
CA
90032, USA. E-mail: gkolesnik@sbcglobal.net}
\Thanks{Research supported in part by the Center for
Advanced Studies
in Mathematics at Ben-Gurion University.}}

\date{}
\maketitle

\begin{abstract}
A natural problem in the context of the coupon collector's problem is the behavior of the maximum of independent geometrically distributed random variables (with distinct parameters). This question has been addressed  by Brennan et al.\ (British J. of Math. \& CS.\:8 (2015), 330--336). Here we provide explicit asymptotic expressions for the moments of that maximum, as well as of the maximum of exponential random variables with corresponding parameters. We also deal with the probability of each of the variables being the maximal one.

The calculations lead to expressions involving Hurwitz's zeta function at certain special points. We find here explicitly the values of the function at these points. Also, the distribution function of the maximum we deal with is closely related to the generating function of the partition function. Thus, our results (and proofs) rely on classical results pertaining to the partition function.
	
	

\end{abstract}

\section{Introduction}
Suppose that a company distributes a commercial product and that each package contains a single coupon. There are $n$ types of coupons, and a customer wants to collect at least one of each. 
We want to know how many packages need to be bought on the average until getting all coupons. This is referred to as the $\textit{coupon collector problem}$. The problem goes back at least as far as de Moivre \citep{Moiver}, who mentioned it in a collection of problems regarding various games of chance.

The expected number of drawings is calculated in a straightforward manner. (Note, though, that if one does not take the right approach, the problem may become quite intricate; see \citep{maunsell}.)  After exactly~$j$ distinct coupons have been seen, the probability of drawing an as yet unseen coupon is $\tfrac{n-j}{n}$. Hence, the number $D_{j}$ of drawings until we see such a new coupon is $G(1-\tfrac{j}{n})$-distributed. The total number of drawings is the sum of all these $D_j$'s. The expected number of drawings is therefore
$nH_n$,
where $H_n$ is the $n$-th harmonic number:
$$H_n=1+\frac{1}{2}+\frac{1}{3}+\cdots +\frac{1}{n}, \qquad n=1,2,3,\ldots.$$
Asymptotically, this expectation is $n(\ln n+\gamma)+O(1)$, where $\gamma = 0.577\ldots$  is the Euler-Mascheroni constant.

The problem, and various extensions thereof, have drawn much attention for many years. Laplace \citep{Laplace}, and also  Erd\H{o}s and  R\'{e}nyi \citep{ErdosRenyi1961}, found the asymptotic distribution of the number of drawings.  Newman and  Shepp \citep{NewmanShepp} considered yet another generalization of the problem, when one wants to collect at least $m$ copies of each coupon. They calculated the asymptotic expected number of required drawings, and Erd\H{o}s and  R\'{e}nyi \citep{ErdosRenyi1961} found the limiting distribution of this quantity. Schelling \citep{VonSchelling1954,VonSchelling1934} and Flajolet et al.\ \citep{FlajoletCoupon} considered the case where various coupons may show up with distinct probabilities.  For other related questions, we refer to \citep{Foata,Holst,Kun,MyersWilf2006,Neal,HofriBoneh}. For more on the history of the problem, see \citep{CouponHistory}.

Our starting point here is Brennan et al.\ \citep{Brennan}. In that paper, the maximum waiting time was considered. That is, let $D_{(n)}$ be the maximum of the variables $D_0,D_1,\ldots,D_{n-1}$, defined above. In \citep{Brennan}, the  expectation, and indeed  all moments, of $D_{(n)}$ have been calculated asymptotically in terms of certain infinite series. For example,
\begin{equation}\label{IntroLabel1}
\begin{aligned}
 E(D_{(n)})&
 =n\sum_{j=1}^{\infty}(-1)^{j+1}\left(\frac{2}{j(3j-1)}+\frac{2}{j(3j+1)}\right)+o(n)
 \approx 1.255n,
\end{aligned}
\end{equation}
and similar formulas hold for all moments.

Our first result is an explicit expression for these infinite series in terms of the values of Hurwitz's zeta function (see, for example, \citep{Apostol,Hurwitz,Lund}) at certain special points. To understand these expressions better, we have calculated these values of  Hurwitz's zeta function in terms of the values of the Riemann zeta function at integer points, and eventually as rational polynomials in $\sqrt{3}$ and $\pi$. In fact, this issue has attracted quite some attention. (See, for example, \citep{CvijovicKlinowski1995,CvijovicKlinowski1999,DwilewiczMinavc2006}.) The values we need have been calculated in \citep{QinShangLi2013}, but here we develop less cumbersome expressions. (See Example \ref{cumbersome} below.)


Following the results of Brennan et al.\ \citep{Brennan}, one may ask: what is the probability of each of the waiting times $D_j$ to be the
longest? Obviously, the last waiting time,~$D_{n-1}$, has the best chance of being the maximal, the second last, $D_{n-2}$, has a smaller chance, etc. The fact that~$E(D_{(n)})\approx1.255n$, which is not much larger than~$E(D_{n-1})$($=n$), hints that the last waiting time has a non-negligible probability of being larger than all its predecessors. Here we find how the asymptotic probability of~$D_{n-m}$ being the maximal one decays as a function of $m$ for large $n$.

Brennan et al.\ \citep{Brennan} note that the limiting distribution of $D_{(n)} / n$ is the distribution of an infinite series of independent exponential random variables with parameters $1,2,3,\ldots$. In this vein, we consider in this paper a continuous version of a coupon collector also. In this version, the interarrival times between coupons are exponentially distributed instead of being fixed. This model has been considered by Boneh and Hofri in a different context \citep{HofriBoneh}.

In Section 2 we define the continuous analogue of the coupon collector problem, and then state our main results in both the discrete and the continuous settings. We also state our results regarding Hurwitz's zeta function. Section 3 presents the calculations leading to explicit expressions for the values of Hurwitz's zeta function at some special points, Section~4~-- the proof of the asymptotic results for the moments, and Section 5 -- the proof of the asymptotic results concerning  the identity of the maximal waiting time.

We would like to express our gratitude to J.-P. Allouche, A. Boneh, M. Hofri, J. Kariv, A. Meurman and A. Yadin for their many helpful comments on the first version of the paper.


\bigskip

\bigskip

\section{Main results}\label{sec:Results}

Recall that the number of coupons  one needs to get, after having already acquired $j$ distinct coupons, until gettting a new one, is a geometric variable with parameter $1-\tfrac{j}{n}$.  We may approximate  this geometric variable by an exponential variable with the same parameter. 
Thus, we define a continuous analogue for the coupon collector problem as follows. There are~$n$ types of coupons, arriving with independent $\textmd{Exp}(1)$-distributed interarrival times. Each coupon has the same probability $\tfrac{1}{n}$ of being of each of the types. The basic question is now about the expected time until all coupon types are obtained. Note that, in the continuous version (unlike the discrete one),  we may alternatively assume that there are $n$ independent flows of coupons, and the interarrival times of each type are $\textmd{Exp}(\tfrac{1}{n})$-distributed.
In the process, the collector gets new coupons $n$ times. Denote the times these new coupons arrive, for both the discrete model and the continuous model, by $T_1<T_2<\cdots <T_n$. We set~$T_0=0$. The waiting times $T_{j+1}-T_j$, $0\le j\le n-1$, between new coupons will be denoted by $D_j$ for the discrete model and by $W_j$ for the continuous model. Thus, $W_j\sim\textmd{Exp}(\tfrac{n-j}{n})$ and, as mentioned above, $D_j\sim G(\tfrac{n-j}{n})$,~ $0\le j\le n-1$.
The main quantity to have been studied in detail is the total time $T_n=\sum_{j=0}^{n-1}D_j$ until the collection is complete (under the discrete model). In this paper, though, our main interest is in  $D_{(n)}=\max_{0\le j\le n-1}D_j$ and $W_{(n)}=\max_{0\le j\le n-1}W_j$. 

We start with the continuous model. Put $X_i=W_{n-i}/n$, so that $X_i\sim\textmd{Exp}(i)$, $1\le i\le n$. 
We have~$W_{(n)}=n\cdot \max\lbrace{X_1,\ldots,X_n\rbrace}$. We want to understand the asymptotic behavior of~$W_{(n)}$. The advantage in passing to the variables $X_i$ is that we may once and for all take an infinite sequence $\left(X_i\right)_{i=1}^{\infty}$ of independent variables $X_i\sim\textmd{Exp}(i)$. Put $M=\max_{1\le i< \infty}X_i$.  (Note that, by Borel-Cantelli's Lemma, $M$ is well defined.) We will start with the moments of $M$.

\begin{theorem}\label{MeanVarOfMax}
The expectation and variance of the maximum waiting time are
$$
E(M)= \frac{4\sqrt{3}}{3}\pi-6\approx 1.255,
$$
and
$$
V(M)=-\frac{28}{3}\pi ^2-16\sqrt{3}\pi+180\approx 0.821,
$$
respectively.
\end{theorem}
The theorem is a special case of Theorem \ref{explicit_k_th_moment} (or Theorem \ref{explicit_k_th_momentFormula2}) below, which gives all moments of $M$. It will be instructive, though, to calculate $E(M)$ separately from $E(M^k)$. We will omit the calculation of $V(M)$.

To express the higher moments of $M$, we need Hurwitz's zeta function~$\zeta(s,a)$, defined by
\begin{equation}\label{def-hurwitz-zeta}
\zeta(s,a)=\sum_{m=0}^{\infty}\frac{1}{\left(m+a\right)^s},\qquad s,a\in\bf{C},\ \Re{}\, \mathit{s}>1,\ \Re\, \mathit{a}>0.
\end{equation}
For more information about Hurwitz's zeta function, we refer the reader to \citep{Apostol,Hurwitz,Lund,Oswald}.

\begin{theorem}\label{explicit_k_th_moment}
For  $k=1,2,\ldots$:
\begin{equation}\label{MometsOfMFormula1}
\begin{aligned}
    E(M^k)
    = k!(-1)^k\:&\left(-6^k \left(\frac{2\pi}{3\sqrt{3}}-1\right)\binom{2k-2}{k-1}\right.\\
    &\:+\sum_{j=1}^{\lfloor k/2\rfloor }2^{k+1}\binom{2k-2j-1}{k-1}3^{k-2j}\zeta(2j)\left(1-2^{1-2j}\right)\\
    &\:+\sum_{j=2}^{k}\binom{2k-j-1}{k-1}\\
    &\qquad \quad  \cdot\left. 6^{k-j}\Big((-1)^j\Big(\zeta(j,\tfrac{1}{3})-\zeta(j,\tfrac{5}{6})\Big)+\zeta(j,\tfrac{2}{3})-\zeta(j,\tfrac{1}{6})+6^j\Big)\right).
\end{aligned}
\end{equation}
\end{theorem}
The right-hand side of (\ref{MometsOfMFormula1}) involves the values of  Hurwitz's zeta function
$\zeta(s,a)$ for positive integers $s$ and $a={1}/{6},{1}/{3},{2}/{3},{5}/{6}$. More specifically, for positive even values of $s$, we need the sums $\zeta(s,1/3)+\zeta(s,2/3)$ and $\zeta(s,1/6)+\zeta(s,5/6)$, which are easy to calculate (using, say, \citep{QinShangLi2013}). However, for odd $s$ we need the differences $\zeta(s,1/3)-\zeta(s,2/3)$ and
 $\zeta(s,1/6)-\zeta(s,5/6)$. Thus, to write
 the right-hand side of \Eq{MometsOfMFormula1} in a more elementary way, we need to find the values of $\zeta(s,a)$ for odd integers and $a={1}/{6},{1}/{3},{2}/{3},{5}/{6}$. These values were given in \citep{QinShangLi2013} as quite cumbersome expressions. Our next result provides more convenient expressions for these values (which, of course, should yield the same results).
 For completeness, we  deal with~$a={1}/{4},{1}/{2},{3}/{4}$ as well, although these are not needed for simplifying Theorem \ref{explicit_k_th_moment}.

\begin{theorem}\label{KolesnikTheorem1} For any $m \geq 1$ we have:
\begin{equation*}
\begin{aligned}
  \zeta(k, 1/2) & = (2^k - 1) \zeta(k),\qquad k = 2m,2m+1; \\
  \zeta(2m + 1, 1/3) & = \frac{3^{2m + 1} - 1}2 \zeta(2m + 1) + \frac{K_m(3)}{\sqrt 3}; \\
  \zeta(2m + 1, 2/3) & = \frac{3^{2m + 1} - 1}2\zeta(2m + 1) - \frac{K_m(3)}{\sqrt 3}; \\
  \zeta(2m + 1, 1/4) & = \left[2^{2m}\left(2^{2m+1} -1\right)\right] \zeta(2m + 1) + \frac12 K_m(4); \\
  \zeta(2m + 1, 3/4) & = \left[2^{2m}\left(2^{2m+1} -1\right) \right] \zeta(2m + 1) - \frac12K_m({4}) \\
  \zeta(2m + 1, 1/6) & = \frac{\left(2^{2m+1}-1\right)\left(3^{2m+1}-1\right)}{2} \zeta(2m + 1) + \frac{K_m(6) - K_m(3)}{\sqrt 3} \\
  \zeta(2m + 1, 5/6) & = \frac{\left(2^{2m+1}-1\right)\left(3^{2m+1}-1\right)}{2} \zeta(2m + 1) - \frac{K_m(6) - K_m(3)}{\sqrt 3} 
\end{aligned}
\end{equation*}
where, for $q\in\lbrace{3,4,6\rbrace}$, 
\begin{equation}\label{Cm}
K_m(q) = \frac{(2\pi )^{2m + 1} (-1)^m}{2(2m)!} \left[\frac q2 - \sum_{j = 0}^{m} B_{2j} \frac{(2m)!\cdot q^{2j}}{(2m -2j+1)!(2j)!} \right].
\end{equation}
\end{theorem}

\begin{remark}
$K_m(q)$ in the theorem coincides with $C_m(a)$, to be used in the proof for $a=1/q$. We have written it here this way to make the expression simpler.
\end{remark}

\begin{example}\label{cumbersome}
The second addend in the first non-trivial value in the theorem, $\zeta(2m + 1, 1/3)$,  is ${K_m(3)}/{\sqrt 3}$, whereas  in \citep{QinShangLi2013} it is:
$$ \frac{\sqrt{3}}{2\pi}
\left(\left(2m+ 2 + 3^{2m+2}\right)\zeta(2m+ 2)-2\sum_{j=0}^{m-1}
3^{2m-2j}\zeta(2m-2j)\zeta (2j + 2)\right).
$$
\end{example}

Employing Theorem \ref{KolesnikTheorem1}, we are able to rewrite Theorem \ref{explicit_k_th_moment} without reference to Hurwitz's zeta function or, indeed, even Riemann's zeta function (at odd integers). We do use, however,  the Bernoulli numbers. We first define the function $B_m(x)$ implicitly by:
$$
\frac{ze^{xz}}{e^z-1}=\sum_{m=0}^{\infty}\frac{B_m(x)}{m!}z^m,\qquad\left|z\right|<2\pi.
$$
The numbers $B_m(0)$ are the $\textit{Bernoulli numbers}$, and are denoted by $B_m$. The Bernoulli numbers are a very classical object in analytic number theory. For further details we refer the reader to
\citep{Apostol,Ireland}.

\begin{theorem}\label{explicit_k_th_momentFormula2}
For $k=1,2,\ldots$, 
\begin{align*}
    E(M^k)= k!\left(-6\right)^k\Bigg(&\binom{2k-1}{k}- \frac{2\pi}{3\sqrt{3}} \binom{2k-2}{k-1}\\
    &\:+\sum_{j=1}^{\lfloor k/2\rfloor }\frac{(-1)^{j}\cdot \pi^{2j}B_{2j}}{(2j)!}\Big(2^{2j-1}-1\Big)\Big(1-\frac{3}{3^{2j}}\Big) \binom{2k-2j-1}{k-1}\\
    &\:-\frac{2}{\sqrt{3}}\sum_{j=1}^{\lfloor (k-1)/2\rfloor}
    \left(\frac{\pi}{3}\right)^{2j+1}
    (-1)^j\binom{2k-2j-2}{k-1}\\
    &\qquad \qquad \qquad \,\,\cdot \left(\frac{3j+1}{(2j+1)!}-3\sum_{\ell=1}^{j}\frac{{6}^{2\ell-1}B_{2\ell}}{(2j-2\ell+1)!(2\ell)!}\right)\Bigg).
\end{align*}where $B_m$ is the  $m$-th Bernoulli  number.
\end{theorem}

In Table \ref{table2} we present some results concerning the first five moments of $M$.
 For  each $k$, we present the  $k$-th moment of $M$ as obtained by Theorem \ref{explicit_k_th_momentFormula2} and as calculated numerically by Brennan et al.\:\citep{Brennan}. One checks easily that the two indeed coincide.

\bigskip
\newcommand\VRule[1][\arrayrulewidth]{\vrule width #1}
\begin{centering}
\setcellgapes{7pt}
\begin{table}[ht]
\makegapedcells
\centering
\begin{tabular}{!{\VRule}c!{\VRule}c!{\VRule}c!{\VRule}c!{\VRule}}
\specialrule{2pt}{0pt}{0pt}
$k$ 
& {$E(M^k)$ by Thm. \ref{explicit_k_th_momentFormula2}}
& {$ E(M^k)$ (Numerical)}  
\\ \hline
& & \\[-2.6ex] \hline
$1$ & $\frac{4\sqrt{3}}{3}\pi-6$  & $1.255$ 
\\ \hline
$2$& $-{4\pi^2}- 32{\sqrt{3}}\pi+216$ & $2.397$  
\\ \hline
$3$ & $\frac{80\sqrt{3}}{3}\pi^3 + 216\pi^2 + {1728 }{\sqrt{3}}\pi - 12960$ & $6.689$  
\\ \hline
$4$ & \makecell{$\qquad\quad\:-\frac{1456 }{5}\pi^4 -{2560 }{\sqrt{3}}\pi^3 - 17280 \pi^2$\\  $ - { 138240 }{\sqrt{3}}\pi + 1088640$}
 & $25.453$  
\\ \hline
$5$ & \makecell{
    $\frac{5440\sqrt{3}}{3}\pi^5 + {43680 }\pi^4+ {288000}{\sqrt{3}}\pi^3$\\$\qquad\quad\:+ {1814400 }\pi^2  + {14515200}{ \sqrt{3}}\pi - 117573120$}
& $123.705$  
\\ \hline
\specialrule{2pt}{0pt}{0pt}
\end{tabular}
\caption{Moments of the maximum waiting time.}
\label{table2}
\end{table} 
\end{centering}

The following proposition shows that, once we know the moments of $M$, we get good estimates for the moments of the maximal waiting time in both finite models.
\begin{proposition}\label{FiniteContinuous1}
For every fixed $k$, as $n\to\infty$:
\begin{description}
    \item{(a)}
    $$E(W_{(n)}^k)=n^kE(M^k)+O\left(1\right).
$$
    \item{(b)}
    $$E(D_{(n)}^k)=n^kE(M^k)+o\left(n^k\right).
$$

\end{description}
\end{proposition}
\begin{remark}
In fact, the error terms in the proposition are much smaller, as one can show employing the proof techniques of Theorem \ref{estimatePmTheorem} below.
\end{remark}
Next, we consider the probability of each $X_m$ in the sequence $\left(X_i\right)_{i=1}^{\infty}$ being the maximum.
\begin{theorem}\label{estimatePmTheorem}
As $m\to\infty$, 
\begin{equation}\label{estimatePmTheoremEQ}
P(X_{m}=M)= \pi\sqrt{2m}\cdot e^{-\pi\sqrt{{2}/{3}}\cdot\sqrt{m}}\cdot (1+o(1)).
\end{equation}
\end{theorem}
The right-hand side of (\ref{estimatePmTheoremEQ}) is very reminiscent of the asymptotic expression for the partition function. Recall that the partition function $p(m)$ counts the various possibilities of representing a positive integer $m$ as a sum of positive integers. The partition function is a classical object in number theory \citep[Chaper 14]{Apostol}. Hardy and Ramanujan \citep{Hardy} showed that, asymptotically as $m\to\infty$: 
$${\displaystyle p(m)= {\frac {1}{4m{\sqrt {3}}}}\cdot e ^{\pi {\sqrt {{2}/{3}}}\cdot\sqrt{m}}}\left(1+o(1)\right).
$$
Tיhus, by Theorem \ref{estimatePmTheorem},  
 $P\left(X_m=M\right)$ is asymptotically the same as ${1}/{p(m)}$ up to a relatively small factor of~$\Theta\left(\sqrt{m}\right).$
This fact, which may seem coincidental, is not surprising once we notice that distribution function of $M$ is intimately related to the generating function of the partition function. (This will be explained later on at the beginning of Section~\ref{section5}.)
The asymptotic behavior of this generating function was investigated by Hardy and Ramanujan \citep{Hardy}, and is the key to the proof of Theorem \ref{estimatePmTheorem}. 
It would be interesting to explain intuitively the proximity of the two quantities $P\left(X_m=M\right)$ and ${1}/{p(m)}$.

In Table \ref{table1} we present some numerical results relating to Theorem \ref{estimatePmTheorem}. For several values of~$m$, we present three quantities:
\begin{description}
\item{(a)} The exact value of $P(X_m=M)$, given by the integral on the right-hand side of (\ref{beta_m label0}) below, calculated numerically by Mathematica. 
\item{(b)} The main term on the right-hand side of (\ref{estimatePmTheoremEQ}). 
\item{(c)}  The value of the integral 
\begin{equation}\label{HardyRamanujan1}
\int\limits_{0}^{1}\frac{mx^{m-1}}{1-x^m}\sqrt{\frac{2\pi}{1-x}}e^{-\tfrac{\pi^2}{6(1-x)}+\tfrac{\pi^2}{12}}dx,
\end{equation}
in which the correct integrand (see the right-hand side of (\ref{beta_m label0}) below) has been replaced by the estimate of Hardy and Ramanujan. The integral has been calculated numerically by Mathematica. It turns out that, although the asymptotics holds only near the point~$1$, the values in (a) and (c) are pretty close. Note that, for $m=1$, the substitution~$y=\sqrt{1/(1-x)}$ leads to an explicit value for the integral (\ref{HardyRamanujan1}),
$$\int\limits_0^1 \frac{1}{1-x}\sqrt{\frac{2\pi}{1-x}}e^{-\tfrac{\pi^2}{6(1-x)}+\tfrac{\pi^2}{12}}dx = \int\limits_1^{\infty} 2\sqrt{{2\pi}}e^{-\tfrac{\pi^2}{6}y^2+\tfrac{\pi^2}{12}}dy =4\sqrt{3}e^{\tfrac{\pi^2}{12}}\left(1-\Phi\left(\frac{\pi\sqrt{3}}{3}\right)\right),$$
but it seems like the integral cannot be calculated explicitly for $m\ge 2$.
\end{description}
\begin{centering}
\setcellgapes{7pt}
\begin{table}[ht]
\makegapedcells
\centering
\begin{tabular}{!{\VRule}c!{\VRule}c!{\VRule}c!{\VRule}c!{\VRule}}
\specialrule{2pt}{0pt}{0pt}
\multicolumn{1}{!{\VRule}c!{\VRule}}{$m$}&
$\quad P(X_m=M)\quad$ & $\quad\pi\sqrt{2m}\cdot e^{-\pi\sqrt{\tfrac{2}{3}}\cdot\sqrt{m}}\qquad$ & $\int\limits_{0}^{1}\frac{mx^{m-1}}{1-x^m}\sqrt{\frac{2\pi}{1-x}}\cdot e^{-\tfrac{\pi^2}{6(1-x)}+\tfrac{\pi^2}{12}}dx$ \\ \hline
& & \\[-2.6ex] \hline
$1$ & $0.516$  & $0.342$ & $0.550$\\ \hline
$2$& $0.213$ & $0.167$  & $0.225$\\ \hline
$3$ & $10.73\cdot 10^{-2}$ & $9.05\cdot 10^{-2}$  & $11.23\cdot 10^{-2}$\\ \hline
$4$ & $5.98\cdot 10^{-2}$ & $5.26\cdot 10^{-2}$  & $6.23\cdot 10^{-2}$\\ \hline
$5$ & $3.55\cdot 10^{-2}$ & $3.21\cdot 10^{-2}$  & $3.68\cdot 10^{-2}$\\ \hline
$10$ & $4.41\cdot 10^{-3}$ & $4.22\cdot 10^{-3}$  & $4.53\cdot 10^{-3}$\\ \hline
$50$ & $4.20\cdot 10^{-7}$ & $4.17\cdot 10^{-7}$  & $4.26\cdot 10^{-7}$\\ \hline
$100$ & $3.24\cdot 10^{-10}$ & $3.22\cdot 10^{-10}$  & $3.27\cdot 10^{-10}$\\
\specialrule{2pt}{0pt}{0pt}
\end{tabular}
\caption{The probability of each variable to be the maximum.}
\label{table1}
\end{table} 
\end{centering}

Similarly to Proposition \ref{FiniteContinuous1}, 
we may   use the results for an infinite sequence of random variables to estimate the  analogous probabilities in the finite versions. (Again, as in Proposition~\ref{FiniteContinuous1}, the error can actually be reduced.)
\begin{proposition}\label{FiniteContinuous2}
For  fixed $m\ge 1$, as $n\to\infty$
\begin{description}
\item{(a)}
$$
P(W_{n-m}=W_{(n)})= P(X_{m}=M)+O\left(\frac 1n\right).
$$
\item{(b)}
$$
P\left(D_{n-m}=D_{(n)}\right)= P(X_{m}=M)+o(1).
$$
\end{description}
\end{proposition}

\bigskip

\bigskip

\section{Hurwitz's zeta function at special points}
Now we want to prove some formulas for Hurwitz's zeta function $z(s,a)$, defined in (\ref{def-hurwitz-zeta}). Not all formulas in this section are new, but we bring them here to make the calculations as self-contained as possible. It is well-known that, similarly to the Riemann zeta function,~$z(s,a)$ is analytic in the entire complex plane with one simple pole at $s=1$.

Define sequence of polynomials $\left(p_j\right)_{j=1}^{\infty}$ and $\left(q_j\right)_{j=1}^{\infty}$ inductively by $p_1(x) = \frac{x - x^2}2$ and
\begin{equation}\label{p_j q_j}
q_j(x) = \int_0^x p_j(t) dt,\qquad  p_{j+1}(x) = q_j(x) - xq_j(1)
\end{equation}
for $j\ge 1$.

\begin{lemma}\label{KolesnikLemma1} Hurwitz's zeta function satisfies the following properties:
\begin{description}
    \item{(a)} $\zeta(0, a) = \dfrac12 - a.$
    \item{(b)} $\zeta(-1, a) = \dfrac{a - a^2}{2} - \dfrac1{12}.$
    \item{(c)} 
    For every $s$ and $k=0,1,2,\ldots$: 
    \begin{equation}\label{ThmKolesnikItemC}
        \begin{aligned}
         \zeta(s, a) = &\:\frac{a^{1-s}}{s-1} + \frac12 a^{-s} + \sum_{j= 1}^{k + 1} s(s+1) \cdots (s + j - 1) q_j(1) a^{-s-j} \\
         &+\: s(s+1) \cdots (s + k + 2) \int_0^1 p_{k + 2}(x) \zeta(s+k+3, x + a)dx
        \end{aligned}
    \end{equation}
    \item{(d)} $\zeta(-k, a) = -\dfrac{a^{k+1}}{k+1} + \frac12 a^k + \sum_{j= 1}^{k}(-1)^j q_j(1)\dfrac{k!}{(k-j)!}a^{k-j},\qquad k\ge 2.$
    \item{(e)}
    For every $s$ and $k=0,1,2,\ldots$:
    \begin{equation}\label{Lemma1Label1}
        \begin{aligned}\zeta(s) =&\: \frac1{s - 1} + \frac12 + \sum_{j = 1}^{k + 1} s(s+1) \cdots (s + j - 1) q_j(1) \\
        &+\: s(s+1) \cdots (s + k + 2) \int_0^1 p_{k + 2}(x) \zeta(s + k + 3, x + 1) dx.\end{aligned}
    \end{equation}
    \item{(f)} For $k=0,1,2,\ldots$:
    \begin{equation}\label{Lemma1Label2}
    \begin{aligned}
    \zeta'(-k) =&\: -\frac1{(k + 1)^2} + \sum_{j = 1}^{k} (-1)^{j+1} \frac{k!}{(k - j)!} q_j(1) \left[H_k-H_{k-j}\right]\\ 
    &\:+ (-1)^k k! \cdot \int_0^1 p_{k + 1}(x) \zeta(2, x+1) dx.
    \end{aligned}
    \end{equation}
    \item{(g)} 
    \begin{equation*}
    \zeta'(0) = -1 + \int_0^1 p_1(x) \zeta(2,x+1) dx.
    \end{equation*}
    \end{description}
\end{lemma}
\begin{proof}
\begin{description}
    \item{$(a)$} Assume first that $\Re \,\mathit{s} > 1$. Using Stieltjes integration, we obtain
\begin{equation}\label{eqI}
\begin{aligned}
  \zeta(s, a) & =  \int_{0^{-}}^{\infty}\frac{d[x]}{(x + a)^s} \\
  & = \int_0^{\infty} \frac{dx}{(x + a)^s} - \int_{0^{-}}^{\infty}\frac{d\{x\}}{(x + a)^s} \\
  & = \left.\frac{(x + a)^{1 - s}}{1 - s}\right|_{x = 0^{-}}^{\infty} - \left.\frac{\{x\}}{(x + a)^{s}}\right|_{x = 0^{-}}^{\infty} - s \int_0^{\infty}\frac{\{x\}dx}{(x + a)^{s + 1}} \\
  & = \frac{a^{1 - s}}{s - 1} + a^{-s} - s\int_0^1 x \zeta(s + 1, x + a)dx 
\end{aligned}
\end{equation}

Since the right-hand side is analytic for $\Re\, \mathit{s} > 0$, this defines $\zeta(s, a)$ for $\Re\, \mathit{s} > 0$.

Multiplying both sides of \Eq{eqI} by $(s - 1)$ and taking the limit as $s \rightarrow 1$, we obtain $\lim_{s \rightarrow 1} \zeta(s, a)(s - 1) = 1$ for any $a > 0$.

Passing to the limit as  ${s \rightarrow 0^{+}}$, we get 
$$
\zeta(0, a) = 1 - a - \int_0^1 xdx = \frac12 - a.
$$

\item{$(b)$} We have 
\begin{equation*}
\begin{aligned}
\int_0^1 x \zeta(s + 1, x + a) dx & =  \sum_{n = 0}^{\infty}\int_0^1 \frac{x dx}{(n + x + a)^{s + 1}} 
= \sum_{n=0}^{\infty}\int_0^1\frac{dx^2/2}{(n + x + a)^{s + 1}} \\
&= \sum_{n = 0}^{\infty} \left[\left.\frac{x^2}{2(n + x + a)^{s + 1}}\right|_{x = 0}^{1} + (s + 1) \int_0^1 \frac{x^2 dx}{2(n + x + a)^{s + 2}}\right]\\
& =  \frac12 \zeta(s+1, a+1) + \frac{s + 1}2 \int_0^1 x^2 \zeta(s + 2, x + a) dx.
\end{aligned}
\end{equation*}
Use \Eq{eqI} with $(s + 1)$ instead of $s$ to replace $\zeta(s + 1, a + 1)$, and the equality
$$\zeta(s + 1, a + 1) = \zeta(s + 1, a) - a^{-s -1}$$
to obtain 
\begin{equation*}
    \begin{aligned}
 \zeta(s, a) =&\: \frac{a ^{1 - s}}{s - 1} + a^{-s} - \frac s2 \left[\frac{a^{-s}}{s} + a^{-s - 1} - a^{-s -1} - (s + 1)\int_0^1 x \zeta(s + 2, x + a) dx\right]\\
 &-\: \frac{s(s + 1)}2 \int_0^1 x^2 \zeta(s+2, x+a)dx\\
 =&\: \frac{a^{1 - s}}{s - 1} + a^{-s} - \frac12a^{-s} + s(s + 1) \int_0^1 \frac{x - x^2}2 \zeta(s + 2, x + a)dx.    
    \end{aligned}
\end{equation*}

We may write the above formula in the form 
\begin{equation}\label{eqII}
  \zeta(s, a) = \frac{a^{1-s}}{s-1} + \frac12 a^{-s} + s(s + 1) \int_0^1p_1(x)\zeta(s + 2, x+a)dx. 
\end{equation}

Using the equality $\lim_{s \rightarrow -1}(s + 1) \zeta(s + 2, x + a) = 1$, and taking the limit as~$s \rightarrow -1$, we obtain
$$
\zeta(-1, a) = -\frac{a^2}2 + \frac a2 - \int_0^1 p_1(x) dx = -\frac{a^2}2 + \frac a2 - \frac1{12}.
$$
\item{$(c)$} We prove (\ref{ThmKolesnikItemC}) by induction:
\begin{equation}\label{eqIII}
\begin{aligned}
  \zeta(s, a) =&\: \frac{a^{1-s}}{s - 1} + \frac12 a^{-s} + \sum_{j = 1}^{k + 1} s(s + 1) \cdots (s + j - 1) q_j(1) a^{-s -j} \\
  &  +\: s(s + 1) \cdots (s + k + 2) \int_0^1 p_{k + 2}(x) \zeta(s + k + 3, x + a) dx. 
 \end{aligned}
\end{equation}
For $k = 0$, by \Eq{eqII}:
\begin{align*}
  \zeta(s, a) =&\: \frac{a^{1 - s}}{s - 1} + \frac12 a^{-s} + s(s + 1) \sum_{n = 0}^{\infty} \left[\left.\frac{q_1(x)}{(n + x + a)^{s + 2}}\right|_{x = 0}^{1}\right.\\
  &+\:\left. (s + 2) \int_0^1 q_1(x) \frac{dx}{(n + x + a)^{s + 3}} \right] \\
  = &\:\frac{a^{1 - s}}{s - 1} + \frac12 a^{-s} + s(s + 1) q_1(1) \left[\zeta(s + 2, a) - a^{-s -2}\right] \\
  &+\: s(s + 1)(s + 2) \int_0^1 q_1(x) \zeta(s + 3, a + x)dx
\end{align*}
Using \Eq{eqI} with $(s + 2)$ instead of $s$, we obtain
\begin{equation*}
\begin{aligned}
  \zeta(s, a)  = &\: \frac{a^{1-s}}{s-1} + \frac12 a^{-s} \nonumber \\
  & +\: s (s + 1) q_1(1) \left[\frac{a^{-s-1}}{(s + 1)} + a^{-s -2} - a^{-s-2} - (s + 2) \int_0^1 x \zeta(s + 3, x + a)dx \right] \nonumber \\
  & +~ s(s + 1)(s + 2) \int_0^1 q_1(x) \zeta(s + 3, x + a) dx \\
  =&\: \frac{a ^{1-s}}{s-1} + \frac12 a^{-s} + sq_1(1) a^{-s -1}\\
  &+\: s(s+1)(s+2)\int_0^1 \left[q_1(x) - xq_1(1)\right] \zeta(s + 3, x + a) dx.
  \end{aligned}
\end{equation*}Since $q_1(x) - xq_1(1) = p_2(x)$, this proves the formula for $k = 0$.

Now  assume the formula to be true for some $k \geq 0$ and prove it for $k + 1$. Just as above, we have
\begin{equation}\label{eqIII2}
    \begin{aligned}
  \int\limits_0^1 p_{k + 2}(x) \zeta(s + k + 3, x + a) dx 
  =&\:  \sum_{n = 0}^{\infty} \int_0^1 \frac{dq_{k + 2}(x)}{(n + x + a)^{s + k + 3}} \\
   =&\:  \sum_{n = 0}^{\infty} \left[\left.\frac{q_{k + 2}(x)}{(n + x + 2)^{s + k + 3}}\right|_{x = 0}^1 \right] \\
   &+~ (s + k + 3)\int_0^1 q_{k + 2}(x) \zeta(s + k + 4, x + a) dx \\
    =&\:  q_{k + 2}(1) \left[ \zeta(s + k + 3, a) - a^{-s -k -3} \right] \\
    &+\: (s + k + 3) \int_0^1 q_{k + 2}(x) \zeta(s + k + 4, x + a) dx.
        \end{aligned}
\end{equation}
Using \Eq{eqI} with $s + k + 3$ instead of $s$, the right-hand side of (\ref{eqIII2}) becomes
\begin{align*}
  q_{k + 2}(1) &\left[\frac{a ^{-s -k -2}}{s + k + 2} - (s + k + 3)\int_0^1 x  \zeta(s + k + 4, x + a)dx\right]\\
  &+ (s + k + 3) \int_0^1 q_{k + 2}(x) \zeta(s + k + 4, x + a)dx \\
    & =  q_{k + 2}(1) \frac{a^{-s-k-2}}{s + k + 2} + (s + k + 3) \int_0^1 p_{k + 3}(x) \zeta(s + k + 4, x + a) dx.
\end{align*}

Substituting in  \Eq{eqIII}, we obtain: 
\begin{equation}\label{eqIV}
\begin{aligned}
  \zeta(s, a)  = &\: \frac{a^{1 - s}}{s - 1} + \frac{a^{-s}}{2} + \sum_{j = 1}^{k + 2} s(s + 1) \cdots (s + j + 1) q_j(1) a^{-s-j}  \\
  & +~ s(s + 1) \cdots (s + k + 3) \int_0^1 p_{k + 3}(x) \zeta(x + k + 4, x + a) dx 
\end{aligned}
\end{equation}
This completes the proof of (\ref{ThmKolesnikItemC}).

\item[$(d)$] The preceding part provides a formula for the analytic continuation of $\zeta(s, a)$ for $\Re\, \mathit{s > -k}-2$. Putting $s = -k$, we obtain
\begin{equation} \label{eqV}
\zeta(-k, a)  = - \frac{a^{k + 1}}{k + 1} + \frac12 a^k + \sum_{j = 1}^{k} (-1)^{j} q_j(1)a^{k-j}\frac{k!}{(k-j)!}.
\end{equation}
\item{$(e)$} Follows from (\ref{eqIV}) by taking  $a = 1$.
\item{(f)} Follows from \Eq{Lemma1Label1} for $s=-k$.
\item{$(g)$} Follows from \Eq{Lemma1Label2} for $k=0$.
\end{description}
\end{proof}

\begin{corollary}
Let $B_n$ be the $n$-th Bernoulli number. Then $B_0 =1$, $B_1 =-1/2$, and for any $n\ge2$ we have 
\begin{equation}\label{qjByBernoulli}
    B_n = (-1)^n q_{n-1}(1) n!.
\end{equation}
\end{corollary}

\begin{proof}
By \citep[Thm. 12.12 and 12.13]{Apostol}
$$\zeta(-n,a) = - \frac{B_{n+1}(a)}{n+1} = -\frac{1}{n+1} \sum_{k=0}^{n+1}  \binom{n+1}{k} \cdot B_k
\cdot a^{n+1-k}.
$$
Comparing this to the formula for $\zeta(-n,a)$ in Lemma \ref{KolesnikLemma1}.(c), we get
\begin{equation*}
    -\frac{1}{n+1}\label{qjByBernoulliLabel2} \sum_{k=0}^{n+1}  \binom{n+1}{k} \cdot B_k
\cdot a^{n+1-k}=-\frac{a^{n+1}}{n+1}+\frac{1}{2} a^n+\sum_{k=1}^n (-1)^k q_k(1)\frac{ n!}{(n-k)!}a^{n-k}.
\end{equation*}
 Since these two sides, considered as polynomials in $a$, are identical, the coefficients of corresponding terms must be equal.  
 Equating the coefficient of $a^{n+1}$ we obtain  $-B_0/(n+1)=-1/(n+1)$, thus, $B_0=1$.  Equating the coefficient of $a^n$ we obtain $ -B_1=1/2$, and thus $B_1=-1/2$. Equating the coefficient of $a^{n+1-k}$ for  $k\ge2$, we obtain $- n!/(k!(n+1-k)!)\cdot B_k=(-1)^{k-1}q_{k-1}(1) n!/(n+1-k)!$, and therefore $B_k =(-1)^k  q_{k-1}(1) \cdot n!$ .
\end{proof}

\begin{Proof}{\ of Theorem \ref{KolesnikTheorem1}} Using the functional equation for Hurwitz's zeta function with~$s = k + 1$ and $k = 2m$, $m \geq 1$, we obtain 
\begin{equation*}
    \begin{aligned}
\zeta(-2m, a) &= \frac{\Gamma(2m+1)}{(2\pi)^{2m + 1}} \left[e^{-\frac{\pi i (2m + 1)}{2}} \sum_{n = 1}^{\infty} \frac{e^{2\pi i n a}}{n^{2m + 1}} + e^{\frac{\pi i (2m + 1)}2} \sum_{n = 1}^{\infty} \frac{e^{-2\pi i n a}}{n^{2m + 1}}\right] \\
&= \frac{2(2m)!(-1)^m}{(2\pi)^{2m + 1}}\sum_{n = 1}^{\infty}\frac{\sin(2\pi n a)}{n^{2m + 1}}.
    \end{aligned}
\end{equation*}
If $a =  p/q$ then, since $C_m(a) = \frac{(2\pi q)^{2m + 1}(-1)^m}{2(2m)!}\zeta(-2m, a)$, we obtain
\begin{equation}\label{eqVII}
\begin{aligned}
  C_m(a) & = \sum_{c = 1}^{q - 1}\sin(2\pi c a) \sum_{n = 0}^{\infty}\frac{q^{2m + 1}}{(nq + c)^{2m + 1}} 
   =  \sum_{c = 1}^{q - 1} \sin(2\pi c a) \zeta\left(2m + 1, \frac cq\right).
\end{aligned}
\end{equation}
Taking $a =  1/q,  2/q, \ldots, {(q - 1)}/q$, we get $q - 1$ equations with $q - 1$ unknowns $\zeta(2m + 1, \sfrac c q)$ $(c = 1, \ldots, q - 1)$ but the coefficient matrix has  rank at most ${(q - 1)}/{2}$.

We have one more equation for every $q$; $$\sum_{c = 1}^{q} \sum_{n = 0}^{\infty} \frac1{(nq + c)^s} = \zeta(s),
$$
and therefore
\begin{equation}\label{eqVIII}
\begin{aligned} 
  \sum_{c = 1}^{q - 1} \zeta\left(s,  \frac cq\right) &=\left(q^s - 1\right)\zeta(s),  \\
  \sum_{c = 1}^{q - 1} \zeta\left(2m + 1,  \frac cq\right)& = \left(q^{2m + 1} - 1\right)\zeta(2m + 1).
\end{aligned}
\end{equation}

Thus, for $q = 2$ we get
\begin{equation}\label{eqIX}
  \zeta\left(k, \frac 12\right) = (2^k - 1) \zeta(k),\qquad k\ge 2.
\end{equation}
For $q = 3$, we get by \Eq{eqVII}
\begin{equation*}
\begin{aligned}
  C_m\left(1/3\right) & = \frac{\sqrt 3}2 \zeta\left(2m + 1, \frac13\right) - \frac{\sqrt 3}2 \zeta\left(2m + 1, \frac23\right),
\end{aligned}
\end{equation*}
and by \Eq{eqVIII}
\begin{equation*}
\begin{aligned}
  C_m\left(1/3\right) & = \frac{\sqrt 3}2 
  (3^{2m + 1} - 1) \zeta(2m + 1) & = \zeta\left(2m + 1, \frac13\right) + \zeta\left(2m + 1, \frac23\right).
\end{aligned}
\end{equation*}
and
\begin{equation}\label{eqX}
\begin{aligned}
  \zeta\left(2m + 1, \frac13\right) & =  \frac{C_m(1/3)}{\sqrt 3} + \frac{3^{2m + 1} - 1}2\zeta(2m + 1), \\
  \zeta\left(2m + 1, \frac23\right) & = \frac{3^{2m + 1} - 1}2\zeta(2m + 1) - \frac{C_m(1/3)}{\sqrt 3}.
\end{aligned}
\end{equation}
For $q = 4$, we get again 2 equations with 2 unknowns: By \Eq{eqVII}
$$C_m\left(\frac14\right) = \zeta\left(2m + 1, \frac14\right) - \zeta\left(2m + 1, \frac{3}{4}\right),
$$
and by \Eq{eqVIII}-\Eq{eqIX}
\begin{equation*}
    \begin{aligned}
\zeta\left(2m + 1, \frac14\right) + \zeta\left(2m + 1, \frac34\right) &= (4^{2m + 1} - 1) \zeta(2m + 1) - \zeta\left(2m + 1, \frac12\right)\\
&= \left(4^{2m + 1} - 2^{2m + 1}\right) \zeta(2m + 1).
    \end{aligned}
\end{equation*}
Hence
$$
\zeta\left(2m + 1,\frac14\right) = \frac12C_m(1/4) + \frac12\left(4^{2m + 1} - 2^{2m + 1}\right)\zeta(2m + 1),
$$
and 
$$
\zeta\left(2m + 1, \frac34\right) = \frac12\left(4^{2m + 1} - 2^{2m + 1}\right)\zeta(2m + 1) - \frac12C_m(1/4).
$$

For $q = 6$, we get 2 equations with 2 unknowns (as $\zeta(2m + 1, 1/3)$ and $\zeta(2m + 1, 1/2)$ and $\zeta(2m + 1, \frac23)$ are known by  \Eq{eqIX} and \Eq{eqX}): 
\begin{equation*}
C_m(1/6) = \sum_{c = 1}^{5} \sin\frac{\pi c}3 \zeta\left(2m + 1, \frac c6\right)
\end{equation*}
and 
$$\sum_{c = 1}^{5}\zeta\left(2m + 1, \frac c6\right) = (6^{2m + 1} - 1)\zeta(2m + 1),$$
i.e., 
\begin{equation}\label{TheoreKolesLabel1}
\frac12 \left[\zeta\left(2m + 1, \frac16\right) + \zeta\left(2m + 1, \frac13\right) - \zeta\left(2m + 1, \frac23\right) - \zeta\left(2m + 1, \frac56\right)\right] = \frac{C_m(1/6)}{\sqrt 3},
\end{equation}
and 
\begin{equation}\label{TheoreKolesLabel2}
\frac12\sum_{c = 1}^{5} \zeta\left(2m + 1, \frac c6\right) = \frac{(6^{2m + 1} - 1)}2\zeta(2m + 1).
\end{equation}

Adding \Eq{TheoreKolesLabel1} and \Eq{TheoreKolesLabel2} by sides,   we get
\begin{equation*}
  \frac{6^{2m + 1} - 1}2 \zeta(2m + 1) + \frac{C_m(1/6)}{\sqrt 3}  = 
  \zeta\left(2m + 1, \frac16\right) + \zeta\left(2m + 1, \frac13\right) + \frac12\zeta\left(2m + 1, \frac12\right),
\end{equation*}
i.e.,
\begin{equation}\label{eqXI}
\begin{aligned}
  \zeta\left(2m + 1, \frac16\right)  =& \: \frac{C_m(1/6)}{\sqrt 3}\\
  &\:+ \left[\frac{6^{2m + 1} - 1}2 - \frac{3^{2m + 1} - 1}2 - \frac{2^{2m + 1} - 1}2\right]\zeta(2m + 1) - \frac{C_m(1/3)}{\sqrt 3}.
\end{aligned}
\end{equation}
Subtracting \Eq{TheoreKolesLabel1} from \Eq{TheoreKolesLabel2}, we obtain
\begin{equation*}
  \zeta\left(2m + 1, \frac23\right) + \zeta\left(2m + 1, \frac56\right) + \frac12\zeta\left(2m + 1, \frac12\right)  = 
  \frac{6^{2m + 1} - 1}2\zeta(2m + 1) - \frac{C_m(1/6)}{\sqrt 3},
\end{equation*}
i.e.,
\begin{equation}  \label{eqXII}
\begin{aligned}
  \zeta\left(2m + 1, \frac56\right) =&\:  \left[\frac{6^{2m + 1} - 1}2 - \frac{3^{2m + 1} - 1}2 - \frac{2^{2m + 1} - 1}2\right]\zeta(2m + 1)\\
  &\:- \frac{C_m(1/6)}{\sqrt 3} + \frac{C_m(1/3)}{\sqrt 3}.
\end{aligned}
\end{equation}
\qed
\end{Proof}


\bigskip

\section{The moments of the maximum}\label{sec:Proofs}
\begin{Proof}{\ of Theorem \ref{MeanVarOfMax}}
Denote by  $F_M$ the distribution function of  $M$. Clearly:
\begin{align*}
F_{{M}}(t)&
=P\left(\max_{1\le m<\infty}X_m\le t \right)=\prod_{m=1}^{\infty}\left( 1-e^{-mt}\right),\qquad t\ge 0.
\end{align*}
As $M$ is  positive,
\begin{equation}\label{Thm2Label1}
\begin{aligned}
E\left(M\right)&=\int_0^{\infty}\left(1-F_{M}(t)\right)dt=\int_0^{\infty}\left(1-\prod_{m=1}^{\infty}\left( 1-e^{-mt}\right)\right)dt\\
&=\int_0^{1}\left(1-\prod_{m=1}^{\infty}\left( 1-x^{m}\right)\right)\frac{dx}{x}.
\end{aligned}
\end{equation}
As noted by Brennan et al.\ \citep{Brennan}, Euler's Pentagonal Theorem (see, for example,  \citep[Cor. 1.7]{penthagonal})  enables us to rewrite the infinite product on the right-hand side of (\ref{Thm2Label1}) as a power series:
\begin{equation}\label{EulerPentagonal1}
\prod_{m=1}^{\infty}(1-x^m)=1+\sum_{m=1}^{\infty}(-1)^m\Big(x^{m(3m-1)/2}+x^{m(3m+1)/2}\Big),\qquad|x|<1.
\end{equation}
Thus, similarly to \citep[Thm. 2.1]{Brennan},
\begin{equation*}
\begin{aligned}
 E(M)&
 =\int_0^1\sum_{m=1}^{\infty}(-1)^{m+1}\Big(x^{m(3m-1)/2}+x^{m(3m+1)/2}\Big)\frac{dx}{x}\\
 &=\sum_{m=1}^{\infty}(-1)^{m+1}\left(\frac{2}{m(3m-1)}+\frac{2}{m(3m+1)}\right)\\
 &=6\sum_{m=1}^{\infty}(-1)^{m+1}\left(\frac{1}{3m-1}-\frac{1}{3m+1}\right).
\end{aligned}
\end{equation*}
Denote
\begin{equation}\label{d2Addend}
\begin{aligned}
   f(x)=\sum_{m=1}^{\infty}\left(\frac{x^{3m-1}}{3m-1}-\frac{x^{3m+1}}{3m+1}\right).
\end{aligned}
\end{equation}
Note that $E(M)=6f(-1).$ Now,
\begin{equation*}
\begin{aligned}
   f'(x)&=\sum_{m=1}^{\infty}\left({x^{3m-2}}-{x^{3m}}\right)=\left(x-x^3\right)\left(1+x^3+x^6+\cdots\right)\\
   &=\frac{x-x^3}{1-x^3}=\frac{1-x^3}{1-x^3}-\frac{1-x}{1-x^3}=1-\frac{1}{1+x+x^2}.
\end{aligned}
\end{equation*}
A simple integration yields
\begin{equation*}
\begin{aligned}
   f(x)=x-\frac{2}{\sqrt{3}}\arctan\left(\frac{2x+1}{\sqrt{3}}\right)+C.
\end{aligned}
\end{equation*}
Putting $f(0)=0$, we get $C=\pi/(3\sqrt{3})$. Thus
\begin{equation}\label{d1Addend}
\begin{aligned}
   f(-1)=-1-\frac{2}{\sqrt{3}}\arctan\left(\frac{-1}{\sqrt{3}}\right)+\frac{\pi}{3\sqrt{3}}
   =\frac{2\sqrt{3}\pi}{9}-1. 
\end{aligned}
\end{equation}
Altogether $E(M)=4\sqrt{3}\pi/3-6$.

\qed

\end{Proof}

The calculation of the infinite series, that will arise in the proof of Theorem \ref{explicit_k_th_moment} requires the partial fraction decomposition of several rational functions.
\begin{lemma}\label{FractionDecomposition} For every $k\ge 1$ and  $a\in \bf R$:
\begin{align*}
\frac{1}{x^k(1+ax)^k}
=&\:\sum_{j=1}^{k}\binom{2k-j-1}{k-1}\frac{(-a)^{k-j}}{x^j}
+\sum_{j=1}^{k}\binom{2k-j-1}{k-1}\frac{(-a)^k}{(1+ax)^j}.
\end{align*}
\end{lemma}

\begin{proof}
By partial fraction decomposition,
\begin{equation*}\label{Lemma7LabelA}
\begin{aligned}
\frac{1}{x^k(1+ax)^k}=&\:\frac{c_{1}}{x}+\frac{c_2}{x^2}+\cdots+\frac{c_{k}}{x^{k}}+\frac{d_1}{1+ax}+\frac{d_2}{(1+ax)^2}+\cdots+\frac{d_{k}}{(1+ax)^{k}},
\end{aligned}
\end{equation*}for suitable constants $c_1,\ldots,c_k,d_1,\ldots,d_k.$
Consider any $c_j$. By  \citep{MaYuWang2014}, it is the  coefficient of~$x^{k-j}$ in the Laurent expansion of~$(1+ax)^{-k}$:
\begin{align*}
    (1+ax)^{-k} & = \sum_{i=0}^{\infty} \binom{-k}{i} (ax)^i=\sum_{i=0}^{\infty} \binom{k+i-1}{k-1} (-a)^i x^i.
\end{align*}
Thus,
\begin{equation}\label{c coefficients}
    c_j=\binom{2k-j-1}{k-1}(-a)^{k-j},\qquad 1\le j\le k.
\end{equation}
Similarly, $d_j$ is the coefficient of  $(1+ax)^{k-j}$ in the Laurent expansion of $x^{-k}$ at $-1/a$. Let~$y=1+ax$, so that~$x=(1-y)/(-a)$. Then
\begin{align*}
    \frac{1}{x^k} & =(-a)^k(1-y)^{-k} = (-a)^k\sum_{i=0}^{\infty} \binom{-k}{i} (-y)^i=(-a)^k\sum_{i=0}^{\infty} \binom{k+i-1}{i} y^i,
\end{align*}
and therefore
\begin{equation}\label{d coefficients}
    d_j=(-a)^k\binom{2k-j-1}{k-1},\qquad 1\le j\le k.
\end{equation}
\end{proof}

\bigskip
\begin{Proof}{\ of Theorem \ref{explicit_k_th_moment}}
Using the same notations as in the proof of Theorem \ref{MeanVarOfMax},  by (\ref{EulerPentagonal1}) and \citep[p.\:150,\:(6.3)]{Feller1970}
\begin{equation*}
\begin{aligned}
E\left(M^k\right)&=\int_0^{\infty}kt^{k-1}\left(1-F_{M}(t)\right)dt\\
&=\sum_{m=1}^{\infty}(-1)^{m+1}\int_0^{\infty}kt^{k-1}\Big(e^{-tm(3m-1)/2}+e^{-tm(3m+1)/2}\Big)dt.
\end{aligned}
\end{equation*}
As  in the proof of \citep[Thm. 4.1]{Brennan}, 
\begin{equation}\label{theorem4label3}
    E(M^k)=k!2^k\sum_{m=1}^{\infty}(-1)^{m+1}\left(\frac{1}{m^k(3m-1)^k}+\frac{1}{m^k(3m+1)^k}\right).
\end{equation}Denote
\begin{equation}\label{BrennanThm41}
    g(k)=\sum_{m=1}^{\infty}(-1)^{m+1}\left(\frac{1}{m^k(3m-1)^k}+\frac{1}{m^k(3m+1)^k}\right).
\end{equation}
By Lemma \ref{FractionDecomposition},
\begin{equation*}
\begin{aligned}
    g(k)&=\sum_{m=1}^{\infty}(-1)^{m+1}\sum_{j=1}^{k}\left(\frac{(-1)^kc_j}{m^j}+\frac{c'_j}{m^j}+\frac{(-1)^kd_j}{(1-3m)^j}+\frac{d'_j}{(1+3m)^j}\right),
\end{aligned}
\end{equation*}
where 
\begin{equation}\label{Theorem4Label14}
\begin{aligned}
c_j&=\binom{2k-j-1}{k-1}3^{k-j}, \qquad c'_j=\binom{2k-j-1}{k-1}(-3)^{k-j},\\ d_j&=\binom{2k-j-1}{k-1}3^k,  \qquad \:\:\:\:d'_j=\binom{2k-j-1}{k-1}(-3)^k.
\end{aligned}
\end{equation}
Thus
\begin{equation}\label{f(k)}
\begin{aligned}
    g(k)=&\sum_{j=1}^{k}\left((-1)^kc_j+c'_j\right)\sum_{m=1}^{\infty}\frac{(-1)^{m+1}}{m^j}\\
    &+\:\sum_{j=1}^{k}(-1)^{k-j}d_j\sum_{m=1}^{\infty}\frac{(-1)^{m+1}}{(3m-1)^j}+\sum_{j=1}^{k}(-1)^k d_j\sum_{m=1}^{\infty}\frac{(-1)^{m+1}}{(3m+1)^j}\\
    =&\sum_{j=1}^{\lfloor k/2\rfloor }2(-1)^kc_{2j}\sum_{m=1}^{\infty}\frac{(-1)^{m+1}}{m^{2j}}\\
    &+\:\sum_{j=1}^{k}\left((-1)^{k} d_j\left((-1)^{j}\sum_{m=1}^{\infty}\frac{(-1)^{m+1}}{(3m-1)^j}+\sum_{m=1}^{\infty}\frac{(-1)^{m+1}}{(3m+1)^j}\right)\right).
\end{aligned}
\end{equation}
For $j\ge1$ by \citep[Thm. 13.11]{Apostol}
\begin{equation}\label{cj s}
\sum_{m=1}^{\infty}\frac{(-1)^{m+1}}{m^{2j}}=\zeta(2j)\left(1-2^{1-2j}\right).    
\end{equation}
Consider the second addend on the right-hand side of (\ref{f(k)}). For $j=1$, by (\ref{d2Addend}) and (\ref{d1Addend})
\begin{equation}
    \begin{aligned}
     \sum_{m=1}^{\infty}\frac{(-1)^{m+1}}{3m-1}-\sum_{m=1}^{\infty}\frac{(-1)^{m+1}}{3m+1}=\frac{1}9\left(2\sqrt{3}\pi-9\right).
    \end{aligned}
\end{equation}
For $j\ge2$
\begin{equation}\label{dTilde}
\begin{aligned}
\sum_{m=1}^{\infty}\frac{(-1)^{m+1}}{(3m-1)^j}&=\frac{1}{2^j}-\frac{1}{5^j}+\frac{1}{8^j}-\frac{1}{11^j}+\cdots =\frac{1}{6^j}\left(\zeta(j,\tfrac{1}{3})-\zeta(j,\tfrac{5}{6})\right),\\
\sum_{m=1}^{\infty}\frac{(-1)^{m+1}}{(3m+1)^j}&=\frac{1}{4^j}-\frac{1}{7^j}+\frac{1}{10^j}-\frac{1}{13^j}+\cdots =\frac{1}{6^j}\left(\zeta(j,\tfrac{2}{3})-\zeta(j,\tfrac{7}{6})\right)\\
&=\frac{1}{6^j}\left(\zeta(j,\tfrac{2}{3})-\zeta(j,\tfrac{1}{6})+6^j\right).
\end{aligned}
\end{equation}

Altogether, by (\ref{f(k)})-(\ref{dTilde})
\begin{equation}\label{Theorem4Label20}
\begin{aligned}
    g(k)= \:&\sum_{j=1}^{\lfloor k/2\rfloor }2(-1)^kc_{2j}\zeta(2j)\left(1-2^{1-2j}\right)-(-1)^kd_1\cdot\frac{1}9\left(2\sqrt{3}\pi-9\right)\\
    &\:+\sum_{j=2}^{k}\frac{(-1)^{k}d_j}{6^j}\Big((-1)^j\left(\zeta(j,\tfrac{1}{3})-\zeta(j,\tfrac{5}{6})\right)+\zeta(j,\tfrac{2}{3})-\zeta(j,\tfrac{1}{6})+6^j\Big).
\end{aligned}
\end{equation}
Now, (\ref{theorem4label3})
-(\ref{Theorem4Label14}) and (\ref{Theorem4Label20}) complete the proof.

\qed

\end{Proof}
\bigskip


\begin{Proof}{\ of Theorem \ref{explicit_k_th_momentFormula2}}Consider the third addend on the right-hand side of (\ref{MometsOfMFormula1}).
By \Eq{eqVIII} and \Eq{eqIX} 
\begin{equation}\label{Lemma8Label1}
\begin{aligned}
    \zeta(s,\tfrac{1}{2})=(2^s-1)\zeta(s)&=6^s \left(\frac{1}{3^{s}}-\frac{1}{6^s}\right)\zeta(s),\\
    \zeta(s,\tfrac{1}{3})+\zeta(s,\tfrac{2}{3})
    =(3^s-1)\zeta(s)&=6^s \left(\frac{1}{2^{s}}-\frac{1}{6^s}\right)\zeta(s).
\end{aligned}
\end{equation}We also have
\begin{equation}\label{Lemma8Label2}
\begin{aligned}
    \zeta(s,\tfrac{1}{6})+\zeta(s,\tfrac{5}{6})
    &=\sum_{i=1}^{5}\zeta(s,\tfrac{i}{6})-\left(\zeta(s,\tfrac{1}{3})+\zeta(s,\tfrac{2}{3})\right)-\zeta(s,\tfrac{1}{2})\\
    &=6^s \left(1-\frac{1}{6^s}\right)\zeta(s)-6^s \left(\frac{1}{2^{s}}-\frac{1}{6^s}\right)\zeta(s)-6^s \left(\frac{1}{3^{s}}-\frac{1}{6^s}\right)\zeta(s)\\
    &=6^s \left(1+\frac{1}{6^{s}}-\frac{1}{2^s}-\frac{1}{3^s}\right)\zeta(s).
\end{aligned}
\end{equation}
Thus, for even~$j$, by (\ref{Lemma8Label1}) and (\ref{Lemma8Label2}),
\begin{equation}\label{Even dj s}
\frac{1}{6^j}\Big(\zeta(j,\tfrac{1}{3})-\zeta(j,\tfrac{5}{6})+\zeta(j,\tfrac{2}{3})-\zeta(j,\tfrac{1}{6})+6^j\Big)
=1+\zeta(j)\Big(\frac{1}{2^{j-1}}+\frac{1}{3^j}-\frac{2}{6^j}-1\Big).
\end{equation}
For odd $j$, by Theorem \ref{KolesnikTheorem1}:
\begin{equation}\label{Odd dj s}
\begin{aligned}
-\zeta(j,\tfrac{1}{3})+\zeta(j,\tfrac{2}{3})+\zeta(j,\tfrac{5}{6})-\zeta(j,\tfrac{1}{6})
=-\frac{2K_{( j-1)/2} (6)}{\sqrt{3}},
\end{aligned}
\end{equation}
where $K_m(q)$ is defined in (\ref{Cm}). 
Using the same notation as in the proof of Theorem \ref{explicit_k_th_moment}, by \Eq{Theorem4Label20}
\begin{equation}\label{SimplificationSecond}
\begin{aligned}
    g(k)= (-1)^k\:&\left(\sum_{j=1}^{\lfloor k/2\rfloor }2c_{2j}\zeta(2j)\left(1-2^{1-2j}\right)-d_1\cdot\frac{2\sqrt{3}\pi}{9}+\sum_{j=1}^k d_j\right. \\
    &\:+\left. \sum_{j=1}^{\lfloor k/2\rfloor}d_{2j}\cdot\zeta(2j)\Big(\frac{2}{2^{2j}}+\frac{1}{3^{2j}}-\frac{2}{6^{2j}}-1\Big)
    -\sum_{j=1}^{\lfloor (k-1)/2\rfloor}\frac{d_{2j+1}}{6^{2j+1}}\frac{2K_j(6)}{\sqrt{3}}\right).
\end{aligned}
\end{equation}
The sum of the first and fourth addends on the right-hand side of (\ref{SimplificationSecond}) yields $\sum_{j=1}^{\lfloor k/2 \rfloor}g_j(k)$, where
\begin{equation*}
\begin{aligned}
g_j(k)=&\:\zeta(2j)\left(2c_{2j}\Big(1-\frac{2}{2^{2j}}\Big)- d_{2j}\Big(1-\frac{2}{2^{2j}}\Big)\Big(1-\frac{1}{3^{2j}}\Big)\right)\\
=&\:\zeta(2j)\Big(1-\frac{2}{2^{2j}}\Big)\left(2 \binom{2k-2j-1}{k-1}3^{k-2j}
- \binom{2k-2j-1}{k-1}3^k\Big(1-\frac{1}{3^{2j}}\Big)\right)\\
=&\:3^k\zeta(2j)\Big(1-\frac{2}{2^{2j}}\Big)\Big(\frac{3}{3^{2j}}-1\Big) \binom{2k-2j-1}{k-1}.
\end{aligned}
\end{equation*}
By, e.g., \citep[p. 266]{Apostol}, for $j\ge1$
$$
\zeta(2j)=\frac{(-1)^{j+1}\cdot 2^{2j-1}\pi^{2j}}{(2j)!}B_{2j},
$$
where $B_m$ is the  $m$-th Bernoulli  number. Thus
\begin{equation}
\begin{aligned}
g_j(k)&=3^k\frac{(-1)^{j+1}\cdot 2^{2j-1}\pi^{2j}}{(2j)!}B_{2j}\left(1-\frac{2}{2^{2j}}\right)\left(\frac{3}{3^{2j}}-1\right) \binom{2k-2j-1}{k-1}\\
&=3^k\frac{(-1)^{j}\cdot \pi^{2j}}{(2j)!}B_{2j}\left(2^{2j-1}-1\right)\left(1-\frac{3}{3^{2j}}\right) \binom{2k-2j-1}{k-1}.
\end{aligned}
\end{equation}
The second addend on the right-hand side of (\ref{SimplificationSecond}) is
\begin{equation}\label{SimpleficationSecond2}
\begin{aligned}
 d_{1}\frac{2\pi}{3\sqrt{3}}&=\frac{2\pi}{3\sqrt{3}} 3^k \binom{2k-1-1}{k-1}=3^{k}\frac{2\pi}{3\sqrt{3}}  \binom{2k-2}{k-1},
\end{aligned}
\end{equation}
and the third is
\begin{equation}\label{SimpleficationSecond1}
\begin{aligned}
\sum_{j=1}^{k} d_{j}&=\sum_{j=1}^{k} \binom{2k-j-1}{k-j}3^k
=3^k\cdot\frac{2k-1}{k}\binom{2k-2}{k-1}=3^k\binom{2k-1}{k}.
\end{aligned}
\end{equation}
The $j$-th term in the last addend on the right-hand side of (\ref{SimplificationSecond}) is the product of two terms. The first, for  $j\ge1$, is
\begin{equation}\label{SimpleficationSecond3}
\begin{aligned}
\frac{d_{2j+1}}{6^{2j+1}}=\frac{1}{6^{2j+1}}\binom{2k-(2j+1)-1}{k-1}3^k=\frac{3^k}{6^{2j+1}}\binom{2k-2j-2}{k-1},
\end{aligned}
\end{equation}
and the second is
\begin{equation}\label{SimpleficationSecond4}
\begin{aligned}
\frac{2K_j({6})}{\sqrt{3}}&=\frac{2}{\sqrt{3}}\frac{(2\pi )^{2j+1}(-1)^j}{2(2j)!}\left(-\frac{1}{2j+1}+\frac{6}{2}-\sum_{\ell=1}^{j}\frac{(2j)!6^{2\ell}B_{2\ell}}{(2j-2\ell+1)!(2\ell)!}\right)\\
&=\frac{2(2\pi )^{2j+1}(-1)^j}{\sqrt{3}}\left(\frac{3j+1}{(2j+1)!}-3\sum_{\ell=1}^{j}\frac{{6}^{2\ell-1}B_{2\ell}}{(2j-2\ell+1)!(2\ell)!}\right).
\end{aligned}
\end{equation}
Therefore, by  (\ref{SimpleficationSecond3}) and (\ref{SimpleficationSecond4}), the last addend on the right-hand side of  (\ref{SimplificationSecond}) is
\begin{equation}\label{SimpleficationSecond5}
\begin{aligned}
    \frac{3^k2}{\sqrt{3}}
    &\sum_{j=1}^{\lfloor (k-1)/2\rfloor}
    \left(\frac{\pi}{3}\right)^{2j+1}
    (-1)^j\binom{2k-2j-2}{k-1}
    \left(\frac{3j+1}{(2j+1)!}-3\sum_{\ell=1}^{j}\frac{{6}^{2\ell-1}B_{2\ell}}{(2j-2\ell+1)!(2\ell)!}\right).
\end{aligned}
\end{equation}
Altogether, by (\ref{SimplificationSecond})-(\ref{SimpleficationSecond5}), we get our claim.

\qed

\end{Proof}

\bigskip
\begin{Proof}{\ of Proposition \ref{FiniteContinuous1}}
\begin{description}
    \item{$(a)$} 
 The idea is to bound the difference between the distribution functions of $W_{(n)}/n$ and $M$. Since  $F_{W_{(n)}/n}(t)=\prod_{j=1}^n(1-e^{-jt})$,
\begin{equation*}
    F_{W_{(n)}/n}(t)-F_{M}(t)=\prod_{j=1}^n(1-e^{-jt})\left(1-\prod_{j=n+1}^{\infty}(1-e^{-jt})\right).
\end{equation*}
At $t=0$ the difference vanishes. For $t>0$,
$$\prod_{j=n+1}^{\infty}(1-e^{-jt})\ge 1-\sum_{j=n+1}^{\infty}e^{-jt}=1-e^{-(n+1)t}\cdot \frac{1}{1-e^{-t}},
$$
and therefore:
\begin{equation}\label{prop5Label1}
    F_{W_{(n)}/n}(t)-F_{M}(t)\le \prod_{j=1}^n(1-e^{-jt})\cdot  \frac{e^{-(n+1)t}}{1-e^{-t}}=\prod_{j=2}^n(1-e^{-jt})\cdot e^{-(n+1)t}.
\end{equation}Thus
\begin{equation}\label{prop5Label2}
    F_{W_{(n)}/n}(t)-F_{M}(t)\le e^{-nt},\qquad t\ge 0.
\end{equation}

Consider the $k$-th moment of $W_{(n)}$. As $W_{(n)}/n$ is non-negative, by  \citep[p.\:150, (6.3)]{Feller1970} and~(\ref{prop5Label2})
\begin{equation*}\label{prop5Label3a}
\begin{aligned}
  E\left(\frac{W_{(n)}^k}{n^k}\right)&=\int_0^{\infty}kt^{k-1}\left(1-F_{X_{(n)}}(t)\right)dt\\
  &=\int_0^{\infty}kt^{k-1}\left(1-F_{M}(t)+O(e^{-nt})\right)dt\\
  &=E(M^k)+k\int_0^{\infty}t^{k-1}\cdot O(e^{-nt})dt.
\end{aligned}
\end{equation*}
Thus, using \citep[p.\:250, (1)]{Apostol}, 
\begin{equation}\label{prop5Label3}
\begin{aligned}
  E(W_{(n)}^k)&=n^kE(M^k)+n^k\cdot O\left(\frac{k\Gamma(k)}{n^k}\right)=n^kE(M^k)+ O\left(1\right).
\end{aligned}
\end{equation}
\item{$(b)$}
By \citep[Thm. 4]{Brennan}, for every $k$
\begin{equation*}
    \begin{aligned}
     \lim_{n\to\infty}\frac{E(D^k_{(n)})}{n^k}&=k!2^k\sum_{m=1}^{\infty}(-1)^{m+1}\left(\frac{1}{m^k(3m-1)^k}+\frac{1}{m^k(3m+1)^k}\right).
    \end{aligned}
\end{equation*}
Thus, (\ref{theorem4label3}) implies that
$$
E(D^k_{(n)})=n^k E(M^k)+o(n^k).
$$  
\end{description}

\qed

\end{Proof}

\bigskip

\section{\large{The asymptotic distribution  of the maximal waiting time}\label{section5}}

In this section we prove Theorem \ref{estimatePmTheorem}. We have to estimate the probability $P(X_{m}=M)$, which we will denote for brevity by   $p_m$. By independence 
\begin{equation}\label{beta_m label0}
\begin{aligned}
    p_m
    &=P\Big(\max\lbrace{X_1,\ldots,X_{m-1},X_{m+1},\ldots\rbrace}\le X_m\Big)\\
    &=\int \limits_{0}^{\infty}me^{-my}\prod_{j\ne m}^{\infty}\left(1-e^{-{j}y}\right)dy
    =\int \limits_{0}^{\infty}me^{-my}\frac{1}{1-e^{-my}}\prod_{j=1}^{\infty}\left(1-e^{-{j}y}\right)dy\\
    &=\int \limits_{1}^{0}\frac{mx^{m}}{1-x^{m}}\prod_{j=1}^{\infty}\left(1-x^{j}\right)(-x^{-1})dx=\int \limits_{0}^{1}\frac{mx^{m-1}}{1-x^{m}}\prod_{j=1}^{\infty}\left(1-x^{j}\right)dx\\
    &=\int \limits_{0}^{1}mx^{m-1}\frac{f(x)}{1-x^{m}}dx,
\end{aligned}
\end{equation}
where  $f:[0,1]\to\bf{R}$ is Euler's function, given by:
$$
f(x)=\prod_{j=1}^{\infty}\left(1-x^{j}\right),\quad 0\le x\le 1.
$$
We will estimate the right-hand side of (\ref{beta_m label0}) by splitting the interval $[0,1]$ into $5$ sub-intervals,~$\left[x_{j-1},x_{j}\right]$, $1\le j\le 5$, and estimating the integral on each sub-interval separately. The endpoints of these intervals are: 
\begin{equation}\label{Theorem9_x}
\begin{aligned}x_0&={0},\\ x_1&={1-c_1/\sqrt{m}},\\ x_2&={1-b/\sqrt{m}-m^{-2/3}},\\ x_3&={1-b/\sqrt{m}+m^{-2/3}},\\ x_4&={1-c_2/\sqrt{m}},\\x_5&=1,
\end{aligned}
\end{equation}where $b=\pi/\sqrt{6}$ and $c_1,c_2$ are constants with  $c_1>2b$ and $0<c_2<b/6$. We estimate the first integral trivially:
\begin{equation}\label{beta_m label2}
    \begin{aligned}
     \int \limits_{0}^{1-c_1/\sqrt{m}}mx^{m-1}\frac{f(x)}{1-x^{m}}dx\le \int \limits_{0}^{1-c_1/\sqrt{m}}mx^{m-1}dx=\left(1-\frac{c_1}{\sqrt{m}}\right)^{m}\le e^{-c_1 \sqrt{m}}.
    \end{aligned}
\end{equation}
The other four integrals are trickier. 
In these integrals we  will replace~$f$ by an estimate thereof  at $1^{-}$. As it turns out, the function~$f$ is a classical object in number theory, as its inverse
     $${\frac  {1}{f(x)}}=\prod_{m=1}^{\infty}\frac{1}{1-x^m}=\sum _{{m=0}}^{\infty }p(m)x^{m}$$
is the generating function of   $p(m)$. In their study   of the partition function,  Hardy and Ramanujan \citep{Hardy} gave the following estimate,  as $x\to 1^-$ (sea also  \citep[p.\:576, (68)]{FlajoletSedgewick2009}): 
\begin{equation}\label{HR}
f(x)= \sqrt{\frac{2\pi}{1-x}}e^{-\tfrac{\pi^2}{6(1-x)}+\tfrac{\pi^2}{12}}\cdot (1+o(1)).
\end{equation}
In the proof of Theorem \ref{estimatePmTheorem} we shall need the next lemma:
\begin{lemma}\label{MaximumOfIm}
For arbitrary fixed positive integer $m$, let  $J_m:[0,1]\to\bf{R}$ be given by:
\begin{equation}\label{LemmaFunction Jm}
J_m(y)=
\begin{cases}
0,& \qquad y=0,\\ \\
\dfrac{(1-y)^{m-1}}{\sqrt{{y}}}e^{-\tfrac{b}{y}}, & \qquad 0< y\le 1,\:(b>0).
\end{cases}
\end{equation}
\begin{description}
\item{(a)} $J_m$ gets its maximum at the point 
\begin{equation}\label{Maclaures Im label0}
    \begin{aligned}
    y_0(m)=\frac{-1/2-b+ \sqrt{4bm+b^2-5b+1/4}}{2m-3}.
     \end{aligned}
\end{equation}In particular, $y_0(m)=\sqrt{b/m}+O(1/m)$ as $m\to\infty$.
\item{(b)} $J_m$ is increasing in $\left[0,y_0(m)\right]$ and decreasing in $\left[y_0(m),1\right]$.
\end{description}
\end{lemma}

\begin{proof}
\begin{description}
\item{(a)} One easily checks that the function is continuous on $[0,1]$. In the open interval we have:
\begin{equation*}
    \begin{aligned}
     \ln J_m(y)={(m-1)}\ln (1-y)-\tfrac{1}{2}\ln y-\frac{b}{y}, \qquad 0<y<1.
    \end{aligned}
\end{equation*}
Thus
\begin{equation}
    \begin{aligned}
     \frac{d}{d y}\ln J_m(y)=-\frac{m-1}{1-y}-\frac{1}{2y}+\frac{b}{y^2}.
    \end{aligned}
\end{equation}
Equating the right-hand side to $0$, we get
$$
y_{1,2}=\frac{-1/2-b\pm \sqrt{4bm+b^2-5b+1/4}}{2m-3}.
$$
As $y_0(m)\in(0,1)$, we obtain (\ref{Maclaures Im label0}).
\item{(b)}  Since $J_m$ is positive in $(0,1)$, and  $J_m(0)=J_m(1)=0$, the claim follows from the preceding part.
\end{description}
\end{proof}

\begin{Proof}{\ of Theorem \ref{estimatePmTheorem}} Denote: 
$$\beta_j(m)=\int\limits_{x_{j-1}}^{x_j}mx^{m-1}\frac{f(x)}{1-x^{m}}dx,\qquad 1\le j\le 5,
$$where the $x_j$'s are defined in (\ref{Theorem9_x}).
Thus, by (\ref{beta_m label2}) we have  $\beta_1(m)\le e^{-c_1\sqrt{m}}$.
Let $J_m$ be as in Lemma \ref{MaximumOfIm}. Denote
$$
I_m(x)=m\sqrt{2\pi}e^{\pi^2/12}J_m(1-x)=mx^{m-1} \sqrt{\frac{2\pi}{1-x}}e^{-\tfrac{\pi^2}{6(1-x)}+\tfrac{\pi^2}{12}}
$$
with $b$ from (\ref{LemmaFunction Jm}) taken as $\pi^2$. Then:
\begin{equation}\label{beta_m label33}
    \begin{aligned}
     \beta_2(m)&=\int\limits_{1-c_1/\sqrt{m}}^{1-b/\sqrt{m}-m^{-2/3}}\frac{mx^{m-1}}{1-x^m}\sqrt{\frac{2\pi}{1-x}}e^{-\tfrac{\pi^2}{6(1-x)}+\tfrac{\pi^2}{12}}\cdot (1+o(1))dx\\
     &=\int\limits_{1-c_1/\sqrt{m}}^{1-b/\sqrt{m}-m^{-2/3}}\frac{ I_m(x)}{1-x^m}\cdot (1+o(1))dx.
    \end{aligned}
\end{equation}
For $ x\le 1-b/\sqrt{m}-m^{-2/3}$
\begin{equation}\label{beta_m label 34}
    1-x^m\ge 1-\left(1-\frac{b}{\sqrt{m}}-\frac{1}{m^{-2/3}}\right)^m\ge 1-e^{-b\sqrt{m}}=1-o(1).
\end{equation}
By Lemma \ref{MaximumOfIm} 
\begin{equation}\label{beta_m label35}
    \begin{aligned}
     \beta_2(m)&\le\frac{c_1-b}{\sqrt{m}}\cdot I_m\left(1-\frac{b}{\sqrt{m}}-\frac{1}{m^{2/3}}\right).
    \end{aligned}
\end{equation}
Now:
\begin{equation}\label{beta2m label1}
    \begin{aligned}
     \ln  I_m\left(1-\frac{b}{\sqrt{m}}-\frac{1}{m^{2/3}}\right)
     =&\:\ln m + (m - 1)\ln \left(1-\frac{b}{\sqrt{m}}-\frac{1}{m^{2/3}}\right)+\frac{1}{2}\ln(2\pi)\\ &-\:\frac{1}{2}\ln\left(1-\left(1-\frac{b}{\sqrt{m}}-\frac{1}{m^{2/3}}\right)\right)+ \frac{\pi^2}{12} \\
     &-\: \frac{{\pi^2}/{6}}{1-\left(1-b/\sqrt{m}-1/m^{2/3}\right)}.
    \end{aligned}
\end{equation}
Let us we estimate the (non-immediate) terms on the right-hand side of \Eq{beta2m label1}. For the second term:
\begin{equation}\label{beta2m label2}
    \begin{aligned}
     \ln \left(1-\frac{b}{\sqrt{m}}-\frac{1}{m^{2/3}}\right)
     &=-\frac{b}{\sqrt{m}}-\frac{1}{m^{2/3}}-\frac{1}{2}\left(\frac{b}{\sqrt{m}}+\frac{1}{m^{2/3}}\right)^2+O(m^{-3/2})\\
     &=-\frac{b}{\sqrt{m}}-\frac{1}{m^{2/3}}-\frac{b^2}{2m}+O(m^{-7/6}).
    \end{aligned}
\end{equation}
The fourth term:
\begin{equation}\label{beta2m label3}
    \begin{aligned}
     \ln\left(\frac{b}{\sqrt{m}}+\frac{1}{m^{2/3}}\right)
     &=\ln\frac{b}{\sqrt{m}}+\ln\left(1+\frac{1}{b{m}^{1/6}}\right)
     =-\frac{1}{2}\ln m+\ln b+O(m^{-1/6}).
    \end{aligned}
\end{equation}
The last term:
\begin{equation}\label{beta2m label4}
    \begin{aligned}
     \frac{{b^2}}{b/\sqrt{m}+1/m^{2/3}}
     &=\frac{{b\sqrt{m}}}{1+1/{(b{m}^{1/6}})}\\
     &=b{\sqrt{m}}\left(1-\frac{1}{b{m}^{1/6}}+\frac{1}{b^2{m}^{1/3}}-\frac{1}{b^3\sqrt{m}}+O(m^{-2/3})\right)\\
     &=b{\sqrt{m}}-{{m}^{1/3}}+\frac{{m}^{1/6}}{b}-\frac{1}{b^2}+O(m^{-1/6}).
    \end{aligned}
\end{equation}
Altogether, by (\ref{beta2m label1})-(\ref{beta2m label4}):
\begin{equation}\label{beta2m label5}
    \begin{aligned}
    \ln  I_m\left(1-\frac{b}{\sqrt{m}}-\frac{1}{m^{2/3}}\right)=&\:\ln m + (m - 1) \left(-\frac{b}{\sqrt{m}}-\frac{1}{m^{2/3}}-\frac{b^2}{2m}+O(m^{-7/6})\right)\\ 
    &+\:\frac{1}{2}\ln(2\pi)-\frac{1}{2}\left(-\frac{1}{2}\ln m+\ln b+O(m^{-1/6})\right)+ \frac{\pi^2}{12}\\
    &-\ \left(b{\sqrt{m}}-{{m}^{1/3}}+\frac{{m}^{1/6}}{b}-\frac{1}{b^2}+O(m^{-1/6})\right)\\
     =&\:\frac{5}{4}\ln m-2b\sqrt{{m}}-\frac{m^{1/6}}{b}-\frac{\ln b}{2}+\frac{\ln (2\pi)}{2}+\frac{1}{b^2}+O(m^{-1/6}).
    \end{aligned}
\end{equation}
\bigskip 
By \Eq{beta_m label35} and (\ref{beta2m label5}):
\begin{equation}\label{beta2m label6}
    \begin{aligned}
     \beta_2(m)&\le\frac{c_1-b}{\sqrt{m}}\cdot  e^{\tfrac{5}{4}\ln m-2b\sqrt{{m}}-\tfrac{m^{1/6}}{b}+o({m^{1/6}})}
     =e^{-2b\sqrt{{m}}-\tfrac{m^{1/6}}{b}+o({m^{1/6}})}=o\left(e^{-2b\sqrt{m}}\right).
    \end{aligned}
\end{equation}

Consider $\beta_3(m)$. Substituting $x=1-b/\sqrt{m}+s/m^{3/4}$ and  using (\ref{beta_m label 34}),  we get
\begin{equation}\label{beta_m label4}
    \begin{aligned}
     \beta_3(m)&=\int\limits_{-m^{1/12}}^{m^{1/12}}{ I_m\left(1-\frac{b}{\sqrt{m}}+\frac{s}{m^{3/4}}\right)}\cdot m^{-3/4}\cdot (1+o(1))ds.
     \end{aligned}
\end{equation}Denoting
$$
\tilde{I}_m(s)={ I_m\left(1-\frac{b}{\sqrt{m}}+\frac{s}{m^{3/4}}\right)},\qquad {-m^{1/12}}\le s\le{m^{1/12}},
$$we have
\begin{equation}\label{beta_m label3}
    \begin{aligned}
     \ln \tilde{I}_m(s)
     =&\:\ln m + (m - 1)\ln \left(1-\frac{b}{\sqrt{m}}+\frac{s}{m^{3/4}}\right)+\frac{1}{2}\ln(2\pi)\\ &-\:\frac{1}{2}\ln\left(1-\left(1-\frac{b}{\sqrt{m}}+\frac{s}{m^{3/4}}\right)\right)+ \frac{\pi^2}{12} - \frac{{b^2}}{1-\left(1-b/\sqrt{m}+s/m^{3/4}\right)}.
    \end{aligned}
\end{equation}
We estimate the terms on the right-hand side of \Eq{beta_m label3}. For the second term:
\begin{equation}\label{Maclaures Im label2}
    \begin{aligned}
     \ln \left(1-\frac{b}{\sqrt{m}}+\frac{s}{m^{3/4}}\right)
     &=-\frac{b}{\sqrt{m}}+\frac{s}{m^{3/4}}-\frac{1}{2}\left(\frac{b}{\sqrt{m}}-\frac{s}{m^{3/4}}\right)^2+O(m^{-3/2})\\
     &=-\frac{b}{\sqrt{m}}+\frac{s}{m^{3/4}}-\frac{b^2}{2m}+O(m^{-7/6}).
    \end{aligned}
\end{equation}
The fourth term:
\begin{equation}\label{Maclaures Im label3}
    \begin{aligned}
     \ln\left(\frac{b}{\sqrt{m}}-\frac{s}{m^{3/4}}\right)
     &=\ln\frac{b}{\sqrt{m}}+\ln\left(1-\frac{s}{b{m}^{1/4}}\right)\\
     &=-\frac{1}{2}\ln m+\ln b+O(m^{-1/6}).
    \end{aligned}
\end{equation}
The last term:
\begin{equation}\label{Maclaures Im label4}
    \begin{aligned}
     \frac{{b^2}}{b/\sqrt{m}-s/m^{3/4}}
     &=\frac{{b\sqrt{m}}}{1-{s}/{(b{m}^{1/4})}}\\
     &=b\sqrt{m}\left(1+\frac{s}{b{m}^{1/4}}+\frac{s^2}{b^2\sqrt{m}}+\frac{s^3}{b^3{m}^{3/4}}+O(m^{-2/3})\right)\\
     &=b\sqrt{m}+{s{m}^{1/4}}+\frac{s^2}{b}+\frac{s^3}{b^2{m}^{1/4}}+O(m^{-1/6}).
    \end{aligned}
\end{equation}
Altogether, by (\ref{beta_m label3})-(\ref{Maclaures Im label4})
\begin{equation}\label{Maclaures Im label5}
    \begin{aligned}
    \ln \tilde{I}_m(s)=&\:\ln m + (m - 1) \left(-\frac{b}{\sqrt{m}}+\frac{s}{m^{3/4}}-\frac{b^2}{2m}+O(m^{-7/6})\right)+\frac{1}{2}\ln(2\pi)\\ &-\:\frac{1}{2}\left(-\frac{1}{2}\ln m+\ln b+O(m^{-1/6})\right)+ \frac{\pi^2}{12}\\
    &-\: \left(b\sqrt{m}+{s{m}^{1/4}}+\frac{s^2}{b}+\frac{s^3}{b^2{m}^{1/4}}+O(m^{-1/6})\right)\\
     =&\: -2b\sqrt{{m}}+\frac{5}{4}\ln m+\frac{1}{4}\ln 24-\frac{s^2}{b} -\frac{s^3}{b^2{m}^{1/4}}+O(m^{-1/6}).
    \end{aligned}
\end{equation}
\bigskip 
By \Eq{beta_m label4}  and (\ref{Maclaures Im label5}):
\begin{equation*}
    \begin{aligned}
     \beta_3(m)&=\int\limits_{-m^{1/12}}^{m^{1/12}}m^{-3/4}e^{\ln \tilde{I}_m(s)}\cdot (1-o(1))ds\\
     &=\sqrt{m}e^{-2b\sqrt{m}}24^{1/4}\int\limits_{-m^{1/12}}^{m^{1/12}}e^{-\tfrac{s^2}{b}-\tfrac{s^3}{b^2m^{1/4}}}\cdot (1-o(1))ds,
    \end{aligned}
\end{equation*}Now,     
\begin{equation*}
\begin{aligned}
\int\limits_{-m^{1/12}}^{m^{1/12}}e^{-\tfrac{s^2}{b}-\tfrac{s^3}{b^2m^{1/4}}}ds\xrightarrow[m\to\infty]{}
\int\limits_{-\infty}^{\infty}e^{-\tfrac{s^2}{b}}ds=\sqrt{b\pi}=\frac{\pi}{6^{1/4}}.
\end{aligned}
\end{equation*}Thus,
\begin{equation}\label{beta_m label5}
    \begin{aligned}
     \beta_3(m)&=\sqrt{m}e^{-2b\sqrt{m}}24^{1/4}\cdot \frac{\pi}{6^{1/4}}\cdot (1+o(1))
     =\pi\sqrt{2m}e^{-2b\sqrt{m}}\cdot (1+o(1)).
    \end{aligned}
\end{equation}

We  bound $\beta_4(m)$  in the same way as we have bounded $\beta_2(m)$ to obtain: 
\begin{equation}\label{beta4m label1}
    \begin{aligned}
     \beta_4(m)&\le e^{-2b\sqrt{{m}}-\tfrac{m^{1/6}}{b}+o({m^{1/6}})}=o\left(e^{-2b\sqrt{m}}\right).
    \end{aligned}
\end{equation}

For $\beta_5(m)$, similarly to \Eq{beta_m label33},
\begin{equation*}
    \begin{aligned}
     \beta_5(m)&=\int \limits_{1-c_2/\sqrt{m}}^{1}mx^{m-1}\frac{f(x)}{1-x^m}dx
    = \int \limits_{1-c_2/\sqrt{m}}^{1}\frac{mx^{m-1}}{1-x^m}\sqrt{\frac{2\pi}{1-x}}e^{-\tfrac{\pi^2}{6(1-x)}+\tfrac{\pi^2}{12}}\cdot (1+o(1))dx\\
    &\le \int \limits_{1-c_2/\sqrt{m}}^{1}\frac{m}{1-x}\sqrt{\frac{2\pi}{1-x}}e^{-\tfrac{{\pi^2}/{12}}{1-x}}\cdot (1+o(1))dx.
    \end{aligned}
\end{equation*}Substituting $t=1/(1-x)$ and recalling that $c_2<\pi/(6\sqrt{6})$, we obtain:
\begin{equation}\label{beta_m label16}
    \begin{aligned}
     \beta_5(m)&\le \int \limits_{\sqrt{m}/c_2}^{\infty}\frac{m\sqrt{2\pi}t^{3/2}}{t^2}e^{-{\tfrac{\pi^2}{12}}t}\cdot (1+o(1))dt
    = m\sqrt{2\pi}\int \limits_{\sqrt{m}/c_2}^{\infty}\frac{e^{-{\tfrac{\pi^2}{12}}t}}{\sqrt{t}}\cdot (1+o(1))dt\\
    &= m\sqrt{2\pi}\int \limits_{\sqrt{m}/c_2}^{\infty}\frac{e^{-{\tfrac{\pi^2}{12}}t}}{\sqrt{\sqrt{m}/c_2}}\cdot (1+o(1))dt
    \le O(m^{3/4})\cdot \frac{12}{\pi^2}e^{-{\tfrac{\pi^2}{{12c_2}}}\sqrt{m}}\cdot (1+o(1))\\
    &\le O(m^{3/4})\cdot \frac{12}{\pi^2}e^{-{\tfrac{3\pi}{{\sqrt{6}}}}\sqrt{m}}\cdot (1+o(1)).
    \end{aligned}
\end{equation}
Altogether, by  (\ref{beta_m label2}), (\ref{beta2m label6}) and (\ref{beta_m label5})-\Eq{beta_m label16}, we obtain our claim.

\qed

\end{Proof}
\bigskip

\begin{Proof}{\ of Proposition \ref{FiniteContinuous2}}
\begin{description}
\item{$(a)$}
Denote  $p_{m,n}=P(W_{n-m}=W_{(n)})$. Similarly to the proof of Theorem  \ref{estimatePmTheorem}:
\begin{equation*}
\begin{aligned}
    p_{m,n}&=\left(W_{n-m}=W_{(n)}\right)=\left(X_{n-m}=\max\lbrace{X_0,\ldots,X_{n-m-1},X_{n-m+1},\ldots,X_{n-1}\rbrace}\right)\\
    &=\int \limits_{0}^{\infty}me^{-my}\frac{1}{1-e^{-my}}\prod_{j=1}^{n}\left(1-e^{-{j}y}\right)dy.
\end{aligned}
\end{equation*}
Similarly to the proof of Proposition \ref{FiniteContinuous1}, we shall be interested in the difference between the distribution functions $F_{W_{(n)}/n}(t)$ and $F_M(t)$, which we  denote by $r(t)$. We have
\begin{equation*}
\begin{aligned}
    p_{m,n}&=\int \limits_{0}^{\infty}me^{-my}\frac{1}{1-e^{-my}}\left(F_M(y)+r(y)\right)dy\\
    &=P(X_{m}=M) +\int \limits_{0}^{\infty}me^{-my}\frac{1}{1-e^{-my}}r(y)dy.
\end{aligned}
\end{equation*}
Denote:
\begin{equation*}
\begin{aligned}
    \varepsilon_m&=\int \limits_{0}^{\infty}me^{-my}\frac{1}{1-e^{-my}}r(y)dy.
\end{aligned}
\end{equation*}By (\ref{prop5Label1}), for $m>1$ 
\begin{equation*}
\begin{aligned}
    \varepsilon_m&\le \int \limits_{0}^{\infty}me^{-my}\frac{1}{1-e^{-my}}\prod_{j=2}^n(1-e^{-jy})\cdot e^{-(n+1)y}dy\\
    &= \int \limits_{0}^{\infty}me^{-my}\prod_{j=2,j\ne m}^n(1-e^{-jy})\cdot e^{-(n+1)y}dy
    \le \int \limits_{0}^{\infty}m e^{-(m+n+1)y}dy=\frac{m}{m+n+1}.
\end{aligned}
\end{equation*}
For $m=1$
\begin{equation*}
    \begin{aligned}
      \varepsilon_1&\le \int \limits_{0}^{\infty}e^{-y}\frac{1}{1-e^{-y}}\prod_{j=2}^n(1-e^{-jy})\cdot e^{-(n+1)y}dy\\&
    = \int \limits_{0}^{\infty}e^{-y}\left(1+e^y\right)\prod_{j=3}^n(1-e^{-jy})\cdot e^{-(n+1)y}dy\\
    &\le \int \limits_{0}^{\infty} \left(e^{-(n+2)y}+ e^{-(n+1)y}\right)dy\le \frac{2}{n+1}.
    \end{aligned}
\end{equation*}
\item{$(b)$}
Denote  
$$p'_{m,n}=P\left(D_{n-m}>\max\lbrace{D_0,\ldots,D_{n-m-1},D_{n-m+1},\ldots,D_{n-1}\rbrace}\right).$$ 
We have:
\begin{equation*}
\begin{aligned}
    p'_{m,n}&=\sum_{j=1}^{\infty}P(D_{n-m}=j)\cdot P\left(\max\lbrace{D_0,\ldots,D_{n-m-1},D_{n-m+1},\ldots,D_{n-1}\rbrace}<j\right)\\
    &=\sum_{j=1}^{\infty}\left(1-\frac{m}{n}\right)^{j-1}\frac{m}{n}\prod_{i=0,i\ne n-m}^{n-1}P\left(D_i<j\right)\\
    &=\sum_{j=1}^{\infty}\left(1-\frac{m}{n}\right)^{j-1}\frac{m}{n}\prod_{i=0,i\ne n-m}^{n-1}\left(1-\left(1-\frac{n-i}{n}\right)^{j-1}\right)\\
     &=\sum_{j=1}^{\infty}\left(1-\frac{m}{n}\right)^{j-1}\frac{m}{n}\prod_{i=1,i\ne m}^{n}\left(1-\left(1-\frac{i}{n}\right)^{j-1}\right)\\
    &=\frac{1}{n}\sum_{j=0}^{\infty}m\left(1-\frac{m}{n}\right)^{j}\prod_{i=1,i\ne m}^{n}\left(1-\left(1-\frac{i}{n}\right)^{j}\right)\\
    &=\lim_{\gamma\to\infty}\frac{1}{n}\sum_{j=0}^{\gamma n}m\left(1-\frac{m}{n}\right)^{j}\prod_{i=1,i\ne m}^{n}\left(1-\left(1-\frac{i}{n}\right)^{j}\right).
\end{aligned}
\end{equation*}
Now,
\begin{equation*}
    \begin{aligned}
     m\left(1-\frac{m}{n}\right)^{j}\prod_{i=1,i\ne m}^{n}\left(1-\left(1-\frac{i}{n}\right)^{j}\right)
     -me^{-m\cdot\tfrac{j}{n}}\prod_{i=1,i\ne m}^{n}\left(1-e^{-i\cdot\tfrac{j}{n}}\right)\xrightarrow[n\to\infty]{}0,
    \end{aligned}
\end{equation*}
where the convergence is uniform over $j\in[0,\gamma n]$. It follows that:
$$
\frac{1}{n}\sum_{j=0}^{\gamma n}m\left(1-\frac{m}{n}\right)^{j}\prod_{i=1,i\ne m}^{n}\left(1-\left(1-\frac{i}{n}\right)^{j}\right)\xrightarrow[n\to\infty]{}
\int\limits_{0}^{\gamma}me^{-mx}\prod_{i=1,i\ne m}^{\infty}\left(1-e^{-ix}\right)dx.
$$
Therefore 
\begin{equation*}
\begin{aligned}
    \lim_{n\to\infty}p'_{m,n}
    = \int\limits_{0}^{\infty}me^{-mx}\prod_{i=1,i\ne m}^{\infty}\left(1-e^{-ix}\right)dx=P(X_{m}=M).
\end{aligned}
\end{equation*}
\end{description}
\qed

\end{Proof}
\newpage

\bibliography{refs}

\begin{thebibliography}{9}

\bibitem{Lund}Gert Almkvist and Arne Meurman. \textit{Values of Bernoulli polynomials and Hurwitz's zeta function at rational points}. The Royal Society of Canada, C. R. Math. Rep. Acad. Sci. Canada 13 (1991) 2--3, 104--108.


\bibitem{penthagonal}George E. Andrews. \textit{The Theory of Partitions}. 
Addison-Wesley Publishing Co., Reading, Mass.-London-Amsterdam, 1976.

\bibitem{Apostol}Tom M. Apostol. \textit{Introduction to Analytic Number Theory}.  
Springer-Verlag, New York-Heidelberg, 1976.


\bibitem{HofriBoneh}Arnon Boneh and Micha Hofri. \textit{The coupon-collector problem revisited---a survey of      engineering problems and computational methods}. Communications in Statistics. Stochastic Models 13 (1997) 1, 39--66.

\bibitem{Brennan}Charlotte A. C. Brennan, Jonathan Kariv and Arnold Knopfmacher. \textit{Longest waiting time in the coupon collector’s problem}. British Journal of Mathematics \& Computer Science 8 (2015) 4, 330--336.

\bibitem{CvijovicKlinowski1995}Djurdje Cvijovi\'{c}  and Jacek Klinowski. \textit{Closed-form summation of some trigonometric series}. Mathematics of Computation 64 (1995) 209, 205--210.

\bibitem{CvijovicKlinowski1999}Djurdje Cvijovi\'{c}  and Jacek Klinowski. \textit{Values of the {L}egendre chi and {H}urwitz zeta functions at
              rational arguments}. Mathematics of Computation 68 (1999) 228, 1623--1630.
              
\bibitem{DwilewiczMinavc2006}Roman Dwilewicz and J\'{a}n Min\'{a}\v{c}. \textit{The {H}urwitz zeta function as a convergent series}. The Rocky Mountain Journal of Mathematics 36 (2006) 4, 1191--1219.

\bibitem{ErdosRenyi1961}Paul Erd\H{o}s and Alfr\'{e}d R\'{e}nyi. \textit{On a classical problem of probability theory}. Magyar Tud. Akad. Mat. Kutat\'{o} Int. K\"{o}zl. 6 (1961), 215--220.

\bibitem{Feller1970}William Feller. \textit{An Introduction to Probability Theory and its Applications}. {V}ol. {II}, Second edition, John Wiley \& Sons, Inc., New York-London-Sydney, 1971. 

\bibitem{CouponHistory}Marco Ferrante and Monica Saltalamacchia. \textit{The coupon collector’s problem}. Materials Matem\`atics (2014)
http://mat.uab.cat/matmat/PDFv2014/v2014n02.pdf.

\bibitem{FlajoletCoupon}Philippe Flajolet,  Dani\`ele Gardy  and Lo\"{y}s Thimonier.
\textit{Birthday paradox, coupon collectors, caching algorithms and  self-organizing search}. Discrete Applied Mathematics. The Journal of Combinatorial
              Algorithms, Informatics and Computational Sciences, 39(1992) 3, 207--229.

\bibitem{FlajoletSedgewick2009}Philippe Flajolet and Robert Sedgewick. \textit{Analytic Combinatorics}, Cambridge University Press, Cambridge 2009.

\bibitem{Foata}Dominique Foata  and Doron Zeilberger.
 \textit{The collector's brotherhood problem using the  {N}ewman-{S}hepp               symbolic method}. Algebra Universalis 49 (2003) 4, 387--395.

\bibitem{Hardy}Godfrey H. Hardy and Srinivasa  Ramanujan. \textit{Asymptotic {F}ormulaae in {C}ombinatory {A}nalysis}. Proceedings of the London Mathematical Society. Second Series, 17 (1918), 75--115.

      
\bibitem{Holst}Lars Holst. \textit{Extreme value distributions for random coupon collector and
  birthday problems}. Extremes. Statistical Theory and Applications in Science, Engineering and Economics 4 (2001) 2, 129--145 (2002).
  
\bibitem{Hurwitz}Adolf Hurwitz.  \textit{Einige Eigenschaften der Dirichletschen Funktionen $F(s)=\sum\left(\frac{D}{n}\right)\frac{1}{n^{-s}}$, die bei der Bestimmung der Klassenanzahlen bin\"{a}rer quadratischer Formen auftreten}. Zeitschrift f\"{u}r Math. und Physik 27 (1882), 86\textendash{}101 (German).

\bibitem{Ireland}Kenneth Ireland and Michael Rosen. \textit{A Classical Introduction to Modern Number Theory}. Graduate Texts in Mathematics 84, Springer-Verlag, New York, 1990. (Second addition)

\bibitem{Kun}N. D. Kan. \textit{The martingale approach to the coupon collection problem}. Rossi\u{\i}skaya Akademiya Nauk. Sankt-Peterburgskoe Otdelenie. 
Matematicheski\u{\i} Institut im. V. A. Steklova. Zapiski Nauchnykh  Seminarov (POMI) 294 (2002) Veroyatn. i Stat. 5, 113--126, 262.
              
\bibitem{Laplace}Pierre-Simon Laplace. \textit{Th\'{e}orie Analytique des Probabilit\'{e}s}. $2$ Vols., 1812, \'{E}ditions Jacques Gabay, Paris, 1995.  

\bibitem{MaYuWang2014}Youneng Ma, Jinhua Yu and Yuanyuan Wang. \textit{Efficient recursive methods for partial fraction expansion of general rational functions}. Journal of Applied Mathematics, vol. 2014, Article ID 895036 (2014),  https://www.hindawi.com/journals/jam/2014/895036/cta/.

\bibitem{maunsell}Frederick George Maunsell. \textit{A Problem in Cartophily}. The Mathematical Gazette,  22 (1938) 251,  328--331.


\bibitem{Moiver}Abraham de Moivre. \textit{De Mensura Sortis seu; de Probabilitate Eventuum in Ludis a Casu Fortuito Pendentibus}. 1710, Philosophical Transactions of the Royal Society of London.

\bibitem{MyersWilf2006}Amy N. Myers and Herbert S. Wilf.
   \textit{Some new aspects of the coupon collector's problem}. SIAM Review 48 (2006) 3, 549--565.
   
\bibitem{Neal}Peter Neal. \textit{The generalised coupon collector problem}. Journal of Applied Probability 45 (2008) 3, 621--629.
   
\bibitem{NewmanShepp}Donald J. Newman and  Lawrence Shepp.\textit{The double dixie cup problem}. The American Mathematical Monthly 67 (1960) 1, 58--61. 
   
\bibitem{Oswald}Nicola Oswald and  J\"{o}rn Steuding. \textit{Aspects of zeta-function theory in the mathematical works of
              {A}dolf {H}urwitz}.
From arithmetic to zeta-functions, Springer, [Cham] (2016), 309--351.   
   

\bibitem{QinShangLi2013}Huizeng Qin, Nina Shang and 
 Aijuan Li. \textit{Some identities on the Hurwitz zeta function and the extended Euler sums}. Integral Transforms and Special Functions 24 (2013) 7, 561--581.
 
 \bibitem{VonSchelling1934}Hermann von Schelling.  \textit{Auf der Spur des Zufalls}. Deutsches Statistisches Zentralblatt, 26  (1934) 137–-146.

\bibitem{VonSchelling1954}Hermann von Schelling. \textit{Coupon collecting for unequal probabilities}. {The American Mathematical Monthly}, {61} (1954), 306--311.


\end{thebibliography}
\bibliographystyle{abbrv}

\end{document}